\documentclass[10pt]{amsart}

\usepackage{amssymb,amsmath,amsthm}
\usepackage[all]{xy}
\usepackage{eucal}

\usepackage{color}

\theoremstyle{plain}
\newtheorem{thm}[subsection]{Theorem}

\newtheorem{prop}[subsection]{Proposition}
\newtheorem{assumption}[subsection]{Basic Assumption}

\theoremstyle{definition}
\newtheorem{defn}[subsection]{Definition}

\theoremstyle{remark}
\newtheorem{rem}[subsection]{Remark}

\makeatletter
\let\c@equation\c@subsection

\makeatother

\newcommand{\ZZ}{{ \mathbb{Z} }}

\newcommand{\ModR}{{ \mathsf{Mod}_\capR }}

\newcommand{\Spectra}{{ \mathsf{Sp}^\Sigma }}

\newcommand{\M}{{ \mathsf{M} }}

\newcommand{\SymArray}{{ \mathsf{SymArray} }}
\newcommand{\SymSeq}{{ \mathsf{SymSeq} }}

\newcommand{\Set}{{ \mathsf{Set} }}

\newcommand{\Lt}{{ \mathsf{Lt} }}
\newcommand{\Alg}{{ \mathsf{Alg} }}

\newcommand{\AlgO}{{ \Alg_\capO }}
\newcommand{\LtO}{{ \Lt_\capO }}

\newcommand{\capX}{{ \mathcal{X} }}
\newcommand{\capY}{{ \mathcal{Y} }}
\newcommand{\capZ}{{ \mathcal{Z} }}
\newcommand{\capO}{{ \mathcal{O} }}
\newcommand{\capR}{{ \mathcal{R} }}
\newcommand{\capA}{{ \mathcal{A} }}
\newcommand{\capB}{{ \mathcal{B} }}
\newcommand{\capP}{{ \mathcal{P} }}

\newcommand{\powerset} {{ \capP }}

\newcommand{\Ev}{{ \mathrm{Ev} }}

\newcommand{\id}{{ \mathrm{id} }}
\newcommand{\op}{{ \mathrm{op} }}
\newcommand{\pr}{{ \mathrm{pr} }}

\newcommand{\Smash}{{ \,\wedge\, }}

\newcommand{\tensor}{{ \otimes }}

\newcommand{\tensorcheck}{{ \check{\tensor} }}

\newcommand{\wequiv}{{ \ \simeq \ }}

\newcommand{\Iso}{{  \ \cong \ }}

\newcommand{\rarrow}{{ \rightarrow }}

\newcommand{\functor}[3]{{ {#1}\colon\thinspace{#2}\rarrow{#3} }}
\newcommand{\function}[3]{{ {#1}\colon\thinspace{#2}\rarrow{#3} }}

\DeclareMathOperator*{\hocolim}{hocolim}

\DeclareMathOperator*{\colim}{colim}
\DeclareMathOperator*{\holim}{holim}

\DeclareMathOperator{\U}{U}

\title[Higher homotopy excision and {B}lakers--{M}assey theorems]{Higher homotopy excision and {B}lakers--{M}assey theorems for structured ring spectra}

\author{Michael Ching}
\author{John E. Harper}

\address{Department of Mathematics, Amherst College, Amherst, MA, 01002, USA}

\email{mching@amherst.edu}

\address{Department of Mathematics, The Ohio State University, Newark, 1179 University Dr, Newark, OH 43055, USA}
\email{harper.903@math.osu.edu}

\begin{document}

\begin{abstract}
Working in the context of symmetric spectra, we prove higher homotopy excision and higher Blakers--Massey theorems, and their duals, for algebras and left modules over operads in the category of modules over a commutative ring spectrum (e.g., structured ring spectra).
\end{abstract}

\maketitle

\section{Introduction}

In this paper we establish the structured ring spectra analogs of Goodwillie's widely exploited and powerful cubical diagram results \cite{Goodwillie_calc2} for spaces. These cubical diagram results are a key ingredient in the authors' homotopic descent results \cite{Ching_Harper} on a structured ring spectra analog of Quillen-Sullivan theory \cite{Quillen_rational, Sullivan_MIT_notes, Sullivan_genetics}. They also establish an important part of the foundations for the theory of Goodwillie calculus in the context of structured ring spectra; see, for instance, Arone and Ching \cite{Arone_Ching}, Bauer, Johnson, and McCarthy \cite{Bauer_Johnson_McCarthy}, Ching \cite{Ching_duality}, Harper and Hess \cite[1.14]{Harper_Hess}, Kuhn \cite{Kuhn_survey}, and Pereira \cite{Pereira_general_context, Pereira_spectral_operad}. For example, it follows from our results that the identity functor on a category of structured ring spectra is analytic in the sense of Goodwillie \cite{Goodwillie_calc2}.

\begin{assumption}
\label{assumption:commutative_ring_spectrum}
From now on in this paper, we assume that $\capR$ is any commutative ring spectrum; i.e., we assume that $\capR$ is any commutative monoid object in the category $(\Spectra,\tensor_S,S)$ of symmetric spectra \cite{Hovey_Shipley_Smith, Schwede_book_project}; here, the tensor product $\tensor_S$ denotes the usual smash product \cite[2.2.3]{Hovey_Shipley_Smith} of symmetric spectra. We work mostly in the category of $\capR$-modules which we denote by $\ModR$.
\end{assumption}

\begin{rem}
Our results apply to many different types of algebraic structures on spectra including (i) associative ring spectra, which we simply call ring spectra, (ii) commutative ring spectra, and (iii) all of the $E_n$ ring spectra for $1\leq n\leq \infty$ that interpolate  between these two extremes of non-commutativity and commutativity. These structures, and many others, are examples of algebras over operads. We therefore work in the following general context: throughout this paper, $\capO$ is an operad in the category of $\capR$-modules (unless otherwise stated), $\AlgO$ is the category of $\capO$-algebras, and $\LtO$ is the category of left $\capO$-modules. 

While $\capO$-algebras are the main objects of interest for most readers, our results also apply in the more general case of left modules over the operad $\capO$; that generalization will be needed elsewhere.
\end{rem}

\begin{rem}
In this paper, we say that a symmetric sequence $X$ of $\capR$-modules is $n$-connected if each $\capR$-module $X[\mathbf{t}]$ is an $n$-connected spectrum. We say that an algebra (resp. left module) over an operad is $n$-connected if the underlying $\capR$-module (resp. symmetric sequence of $\capR$-modules) is $n$-connected, and similarly for operads. Similarly, we say that a map $X \rarrow Y$ of symmetric sequences is $n$-connected if each map $X[\mathbf{t}] \rarrow Y[\mathbf{t}]$ is an $n$-connected map of $\capR$-modules, and a map of $\capO$-algebras (resp. left $\capO$-modules) is $n$-connected if the underlying map of $\capR$-modules (resp. symmetric sequences) is $n$-connected.
\end{rem}

The main results of this paper are Theorems~\ref{thm:higher_blakers_massey} and \ref{thm:higher_dual_blakers_massey}, which are the analogs of Goodwillie's higher Blakers-Massey theorems \cite[2.5 and 2.6]{Goodwillie_calc2}. These results include various interesting special cases which we now highlight.

One such case is given by the homotopy excision result of Theorem~\ref{thm:homotopy_excision} below. In this $2$-cubical diagram situation, (i) Goerss-Hopkins \cite[2.3.13]{Goerss_Hopkins_moduli} prove a closely related homotopy excision result in the special case of simplicial algebras over an $E_\infty$ operad and remark that it is true more generally for any simplicial operad \cite[2.3.14]{Goerss_Hopkins_moduli}, (ii) Baues \cite[I.C.4]{Baues_combinatorial} proves a homotopy excision result in an algebraic setting that includes simplicial associative algebras, (iii) Schwede \cite[3.6]{Schwede_algebraic} proves a result that is very nearly homotopy excision in the context of algebras over a simplicial theory, and (iv) Dugger-Shipley \cite[2.3]{Dugger_Shipley} prove a homotopy excision result for associative ring spectra that our result recovers as a very special case. 

\begin{thm}[Homotopy excision for structured ring spectra]
\label{thm:homotopy_excision}
Let $\capO$ be an operad in $\capR$-modules. Let $\capX$ be a homotopy pushout square of $\capO$-algebras (resp. left $\capO$-modules) of the form
\begin{align*}
\xymatrix{
  \capX_\emptyset\ar[r]\ar[d] & \capX_{\{1\}}\ar[d]\\
  \capX_{\{2\}}\ar[r] & \capX_{\{1,2\}}
}
\end{align*}
Assume that $\capR,\capO,\capX_\emptyset$ are $(-1)$-connected. Consider any $k_1,k_2\geq -1$. If each $\capX_\emptyset\rarrow \capX_{\{i\}}$ is $k_i$-connected ($i=1,2$), then
\begin{itemize}
\item[(a)] $\capX$ is $l$-cocartesian (Definition \ref{defn:cofibration_cubes_etc}) in $\ModR$ (resp. $\SymSeq$) with $l=k_1+k_2 +1$,
\item[(b)] $\capX$ is $k$-cartesian (Definition \ref{defn:fibration_cubes_etc}) with $k=k_1+k_2$.
\end{itemize}
\end{thm}

Relaxing the assumption in Theorem~\ref{thm:homotopy_excision} that $\capX$ is a homotopy pushout square, we obtain the following result which is the direct analog for structured ring spectra of the original Blakers-Massey Theorem for spaces.

\begin{thm}[Blakers-Massey theorem for structured ring spectra]
\label{thm:blakers_massey}
Let $\capO$ be an operad in $\capR$-modules. Let $\capX$ be a commutative square of $\capO$-algebras (resp. left $\capO$-modules) of the form
\begin{align*}
\xymatrix{
  \capX_\emptyset\ar[r]\ar[d] & \capX_{\{1\}}\ar[d]\\
  \capX_{\{2\}}\ar[r] & \capX_{\{1,2\}}
}
\end{align*}
Assume that $\capR,\capO,\capX_\emptyset$ are $(-1)$-connected. Consider any $k_1,k_2\geq -1$, and  $k_{12}\in\ZZ$. If each $\capX_\emptyset\rarrow \capX_{\{i\}}$ is $k_i$-connected ($i=1,2$) and $\capX$ is $k_{12}$-cocartesian, then $\capX$ is $k$-cartesian, where $k$ is the minimum of $k_{12}-1$ and $k_{1}+k_{2}$.
\end{thm}

The following higher homotopy excision result lies at the heart of this paper. It can be thought of as a structured ring spectra analog of higher homotopy excision (see Goodwillie \cite[2.3]{Goodwillie_calc2}) in the context of spaces. This result also implies that the identity functors for $\AlgO$ and $\LtO$ are $0$-analytic in the sense of \cite[4.2]{Goodwillie_calc2}.

\begin{thm}[Higher homotopy excision for structured ring spectra]
\label{thm:higher_homotopy_excision}
Let $\capO$ be an operad in $\capR$-modules and $W$ a nonempty finite set. Let $\capX$ be a strongly $\infty$-cocartesian (Definition \ref{defn:cofibration_cubes_etc}) $W$-cube of $\capO$-algebras (resp. left $\capO$-modules). Assume that $\capR,\capO,\capX_\emptyset$ are $(-1)$-connected. Let $k_i\geq -1$ for each $i\in W$. If each $\capX_\emptyset\rarrow\capX_{\{i\}}$ is $k_i$-connected ($i\in W$),  then
\begin{itemize}
\item[(a)] $\capX$ is $l$-cocartesian in $\ModR$ (resp. $\SymSeq$) with $l=|W|-1+\sum_{i\in W}k_i$,
\item[(b)] $\capX$ is $k$-cartesian with $k=\sum_{i\in W}k_i$.
\end{itemize}
\end{thm}

The preceding results are all special cases of the following theorem which relaxes the assumption in Theorem \ref{thm:higher_homotopy_excision} that $\capX$ is strongly $\infty$-cocartesian. This result is a structured ring spectra analog of Goodwillie's higher Blakers-Massey theorem for spaces \cite[2.5]{Goodwillie_calc2}. 

\begin{thm}[Higher Blakers-Massey theorem for structured ring spectra]
\label{thm:higher_blakers_massey}
Let $\capO$ be an operad in $\capR$-modules and $W$ a nonempty finite set. Let $\capX$ be a $W$-cube of $\capO$-algebras (resp. left $\capO$-modules). Assume that $\capR,\capO,\capX_\emptyset$ are $(-1)$-connected, and suppose that
\begin{itemize}
\item[(i)] for each nonempty subset $V\subset W$, the $V$-cube $\partial_\emptyset^V\capX$ (formed by all maps in $\capX$ between $\capX_\emptyset$ and $\capX_V$) is $k_V$-cocartesian,
\item[(ii)] $-1\leq k_{U}\leq k_V$ for each $U\subset V$.
\end{itemize}
Then $\capX$ is $k$-cartesian, where $k$ is the minimum of $-|W|+\sum_{V\in\lambda}(k_V+1)$ over all partitions $\lambda$ of $W$ by nonempty sets.
\end{thm}

For instance, when $n=3$, $k$ is the minimum of
\begin{align*}
  k_{\{1,2,3\}}-2,\quad
  &k_{\{1,2\}}+k_{\{3\}}-1,\\
  &k_{\{1,3\}}+k_{\{2\}}-1,\\
  &k_{\{2,3\}}+k_{\{1\}}-1,\quad
  k_{\{1\}}+k_{\{2\}}+k_{\{3\}}.
\end{align*}

Our other results are dual versions of Theorems~\ref{thm:homotopy_excision}, \ref{thm:blakers_massey}, \ref{thm:higher_homotopy_excision} and \ref{thm:higher_blakers_massey}.

\begin{thm}[Dual homotopy excision for structured ring spectra]
\label{thm:dual_homotopy_excision}
Let $\capO$ be an operad in $\capR$-modules. Let $\capX$ be a homotopy pullback square of $\capO$-algebras (resp. left $\capO$-modules) of the form
\begin{align*}
\xymatrix{
  \capX_\emptyset\ar[r]\ar[d] & \capX_{\{1\}}\ar[d]\\
  \capX_{\{2\}}\ar[r] & \capX_{\{1,2\}}
}
\end{align*}
Assume that $\capR,\capO,\capX_\emptyset$ are $(-1)$-connected. Consider any $k_1,k_2\geq -1$. If $\capX_{\{2\}}\rarrow \capX_{\{1,2\}}$ is $k_1$-connected and $\capX_{\{1\}}\rarrow \capX_{\{1,2\}}$ is $k_2$-connected, then $\capX$ is $k$-cocartesian with $k=k_{1}+k_{2}+2$.
\end{thm}

The following result relaxes the assumption that $\capX$ is a homotopy pullback square.

\begin{thm}[Dual Blakers-Massey theorem for structured ring spectra]
\label{thm:dual_blakers_massey}
Let $\capO$ be an operad in $\capR$-modules. Let $\capX$ be a commutative square of $\capO$-algebras (resp. left $\capO$-modules) of the form
\begin{align*}
\xymatrix{
  \capX_\emptyset\ar[r]\ar[d] & \capX_{\{1\}}\ar[d]\\
  \capX_{\{2\}}\ar[r] & \capX_{\{1,2\}}
}
\end{align*}
Assume that $\capR,\capO,\capX_\emptyset$ are $(-1)$-connected. Consider any $k_1,k_2,k_{12}\geq -1$ with $k_1\leq k_{12}$ and $k_2\leq k_{12}$. If $\capX_{\{2\}}\rarrow \capX_{\{1,2\}}$ is $k_1$-connected, $\capX_{\{1\}}\rarrow \capX_{\{1,2\}}$ is $k_2$-connected, and $\capX$ is $k_{12}$-cartesian, then $\capX$ is $k$-cocartesian, where $k$ is the minimum of $k_{12}+1$ and $k_{1}+k_{2}+2$.
\end{thm}

\begin{thm}[Higher dual homotopy excision for structured ring spectra]
\label{thm:higher_dual_homotopy_excision}
Let $\capO$ be an operad in $\capR$-modules and $W$ a finite set with $|W|\geq 2$. Let $\capX$ be a strongly $\infty$-cartesian (Definition \ref{defn:fibration_cubes_etc}) $W$-cube of $\capO$-algebras (resp. left $\capO$-modules). Assume that $\capR,\capO,\capX_\emptyset$ are $(-1)$-connected. Let $k_i\geq -1$ for each $i\in W$. If each $\capX_{W-\{i\}}\rarrow\capX_W$ is $k_i$-connected ($i\in W$),  then $\capX$ is $k$-cocartesian with $k=|W|+\sum_{i\in W}k_i$.
\end{thm}

The last three results are all special cases of the following theorem which is a structured ring spectra analog of Goodwillie's higher dual Blakers-Massey theorem for spaces \cite[2.6]{Goodwillie_calc2}. This specializes to the higher dual homotopy excision result (Theorem \ref{thm:higher_dual_homotopy_excision}) in the special case that $\capX$ is strongly $\infty$-cartesian, and to Theorem~\ref{thm:dual_blakers_massey} in the case $|W| = 2$.

\begin{thm}[Higher dual Blakers-Massey theorem for structured ring spectra]
\label{thm:higher_dual_blakers_massey}
Let $\capO$ be an operad in $\capR$-modules and $W$ a nonempty finite set. Let $\capX$ be a $W$-cube of $\capO$-algebras (resp. left $\capO$-modules). Assume that $\capR,\capO,\capX_\emptyset$ are $(-1)$-connected, and suppose that
\begin{itemize}
\item[(i)] for each nonempty subset $V\subset W$, the $V$-cube $\partial_{W-V}^W\capX$ (formed by all maps in $\capX$ between $\capX_{W-V}$ and $\capX_W$) is $k_V$-cartesian,
\item[(ii)] $-1\leq k_{U}\leq k_V$ for each $U\subset V$.
\end{itemize}
Then $\capX$ is $k$-cocartesian, where $k$ is the minimum of $k_W+|W|-1$ and $|W|+\sum_{V\in\lambda}k_V$ over all partitions $\lambda$ of $W$ by nonempty sets not equal to $W$.
\end{thm}

In this paper, the homotopy groups $\pi_*Y$ of a symmetric spectrum $Y$  denote the \emph{derived} homotopy groups (or true homotopy groups) \cite{Schwede_book_project, Schwede_homotopy_groups}; i.e., $\pi_*Y$ always denotes the homotopy groups of a stable fibrant replacement of $Y$, and hence of a flat stable fibrant replacement of $Y$. See Schwede \cite{Schwede_homotopy_groups} for several useful properties enjoyed by the true homotopy groups of a symmetric spectrum.

\begin{rem}[The chain complexes setting]
The main results in \cite{Harper_Bar} in the context of symmetric spectra are developed side-by-side with the corresponding results in the context of unbounded chain complexes over a commutative ring. The reader of \cite{Harper_Bar} will immediately observe that the proofs of the main results, Theorems \ref{thm:homotopy_excision} to \ref{thm:higher_dual_blakers_massey}, remain true without changes in the unbounded chain complex setting described in \cite{Harper_Bar}, provided that the operad is furthermore $\Sigma$-cofibrant and admissible.  In other words, the proofs in the chain complex setting are precisely identical, one simply replaces homotopy groups with homology groups.

Alternately, one can understand our results in the chain complexes setting by simply appealing to the work of Richter-Shipley \cite{Richter_Shipley}, Schwede-Shipley \cite{Schwede_Shipley_equivalences}, and Shipley \cite{Shipley_comm_ring} to establish zig-zags of Quillen equivalences between the category of $\capO$-algebras in chain complexes and the category of algebras over the corresponding operad in the symmetric spectra setting.
\end{rem}

\begin{rem}[The motivic setting]
It seems reasonable that one could adjust the arguments in this paper to the setting of $t$-structures of Beilinson-Bernstein-Deligne \cite{Beilinson_Berstein_Deligne}, with the idea that one could apply the resulting theorems to the motivic setting, with its variety of $t$-structures; the desirability of Goodwillie calculus \cite{Goodwillie_calc2} for motivic spectra has already been pointed out by Dundas-R{\"o}ndigs-{\O}stv{\ae}r \cite{Dundas_Rondigs_Ostvaer}. The results of the current paper would form part of the foundation for such a theory. It seems reasonable that one could work in the setting of a stable model category with an appropriately compatible symmetric monoidal structure and $t$-structure. Working out the precise conditions and details is beyond the scope of the current paper and will not be explored here.
\end{rem}

\begin{rem}[The $\infty$-categorical setting]
Since the positive flat stable model structure is precisely the condition that guarantees that the underlying $\infty$-category of the model category of $\capO$-algebras is equivalent to the $\infty$-category of (homotopy coherent) algebras over the nerve of $\capO$ defined in Lurie \cite{Lurie_higher_algebra}, it seems reasonable that one could adjust the arguments in this paper to the context of $\infty$-operads \cite{Lurie_higher_algebra}. This is beyond the scope of the current paper and will not be explored here.
\end{rem}

\subsection{Organization of the paper}

In Section \ref{sec:preliminaries} we recall some preliminaries on algebras and modules over operads. In Section \ref{sec:cubical_diagrams} we prove our main results. Much of the work is concerned with proving higher homotopy excision (Theorem~\ref{thm:higher_homotopy_excision}) which we obtain as a special case of a more general result, Theorem \ref{thm:pushout_cofibration_cube_homotopical_analysis}. We then use an induction argument due to Goodwillie to pass from this to the higher Blakers-Massey result (Theorem~\ref{thm:higher_blakers_massey}). We can then use higher Blakers-Massey to deduce, first, higher dual homotopy excision (Theorem~\ref{thm:higher_dual_homotopy_excision}) and then higher dual Blakers-Massey (Theorem~\ref{thm:higher_dual_blakers_massey}).

\subsection*{Acknowledgments}

The second author would like to thank Greg Arone, Kristine Bauer, Bjorn Dundas, Bill Dwyer, Brenda Johnson, Nick Kuhn, Ib Madsen, Jim McClure, and Donald Yau for useful remarks. The second author is grateful to Dmitri Pavlov and Jakob Scholbach for helpful comments that directly led to \cite{Harper_Spectra_Corrigendum}, and to Mark Behrens and Haynes Miller for a stimulating and enjoyable visit to the Department of Mathematics at the Massachusetts Institute of Technology in summer 2011, and for their invitation which made this possible. The authors would like to thank an anonymous referee for his or her suggestions and comments, which have resulted in a significant improvement. The first author was partially supported by National Science Foundation Grant DMS-1144149.

\section{Preliminaries}
\label{sec:preliminaries}

The purpose of this section is to recall various preliminaries on algebras and left modules over operads needed in this paper. Define the sets $\mathbf{n}:=\{1,\dots,n\}$ for each $n\geq 0$, where $\mathbf{0}:=\emptyset$ denotes the empty set. If $W$ is a finite set, we denote by $|W|$ the number of elements in $W$. For a more detailed development of the material in this section, see \cite{Harper_Spectra, Harper_Modules}.

\begin{defn}
\label{defn:symmetric_sequences}
Let $\M$ be a category with initial object $\emptyset$ and $n\geq 0$.
\begin{itemize}
\item $(\Sigma,\amalg,\mathbf{0})$ is the symmetric monoidal category of finite sets and their bijections; here, $\amalg$ denotes disjoint union of sets.
\item $(\ModR,\Smash,\capR)$ is the closed symmetric monoidal category of $\capR$-modules.
\item A \emph{symmetric sequence} in $\ModR$ (resp. $\M$) is a functor $\functor{A}{\Sigma^{\op}}{\ModR}$ (resp. $\functor{A}{\Sigma^{\op}}{\M}$). Denote by $\SymSeq$ the category of symmetric sequences in $\ModR$ and their natural transformations.
\item A symmetric sequence $A$ is \emph{concentrated at $n$} if $A[\mathbf{u}]=\emptyset$ for all $u\neq n$; here, $\emptyset$ denotes the initial object in $\ModR$ (resp. $\M$).
\end{itemize}
\end{defn}

\begin{defn} Let $A_1,\dotsc,A_t\in\SymSeq$. Their \emph{tensor product} $A_1\tensorcheck\dotsb\tensorcheck A_t\in\SymSeq$ is the left Kan extension of objectwise smash along coproduct of sets
\begin{align*}
\xymatrix{
  (\Sigma^{\op})^{\times t}
  \ar[rr]^-{A_1\times\dotsb\times A_t}\ar[d]^{\coprod} & &
  (\ModR)^{\times t}\ar[r]^-{\Smash} & \ModR \\
  \Sigma^{\op}\ar[rrr]^{A_1\tensorcheck\dotsb\tensorcheck
  A_t}_{\text{left Kan extension}} & & & \ModR
}
\end{align*}
\end{defn}

If $X$ is a finite set and $A$ is an object in $\ModR$, we use the usual dot notation $A\cdot X$ (see Mac Lane \cite{MacLane_categories} or \cite[2.3]{Harper_Modules}) to denote the copower $A\cdot X$ defined by
$
  A\cdot X := \coprod_X A
$,
the coproduct in $\ModR$ of $|X|$ copies of $A$. Recall the following useful calculations for tensor products.

\begin{prop}
Let $A_1,\dotsc,A_t\in\SymSeq$ and $U\in\Sigma$, with $u:=|U|$. There are natural isomorphisms
\begin{align}
  \notag
  (A_1\tensorcheck\dotsb\tensorcheck A_t)[U]&\Iso\
  \coprod_{\substack{\function{\pi}{U}{\mathbf{t}}\\ \text{in $\Set$}}}
  A_1[\pi^{-1}(1)]\Smash\dotsb\Smash
  A_t[\pi^{-1}(t)]\\
  \label{eq:tensor_check_calc}
  &\Iso
  \coprod_{u_1+\dotsb +u_t=u}A_1[\mathbf{u_1}]\Smash\dotsb\Smash
  A_t[\mathbf{u_t}]\underset{{\Sigma_{u_1}\times\dotsb\times
  \Sigma_{u_t}}}{\cdot}\Sigma_{u}
\end{align}
\end{prop}

Here, $\Set$ is the category of sets and their maps, and \eqref{eq:tensor_check_calc} displays the tensor product $(A_1\tensorcheck\dotsb\tensorcheck A_t)[U]$ as a coproduct of $\Sigma_{u_1}\times\dotsb\times\Sigma_{u_t}$-orbits. It will be conceptually useful to extend the definition of tensor powers $A^{\tensorcheck t}$ to situations in which the integers $t$ are replaced by a finite set $T$.

\begin{defn}
Let $A\in\SymSeq$ and $T,U\in\Sigma$. The \emph{tensor powers} $A^{\tensorcheck T}\in\SymSeq$ are defined objectwise by
\begin{align*}
  (A^{\tensorcheck\emptyset})[U]:=
  \coprod_{\substack{\function{\pi}{U}{\emptyset}\\ \text{in $\Set$}}}
  \capR,\quad\quad
  &(A^{\tensorcheck T})[U]:=
  \coprod_{\substack{\function{\pi}{U}{T}\\ \text{in $\Set$}}}
  \bigwedge_{t\in T} A[\pi^{-1}(t)]\quad(T\neq\emptyset).
\end{align*}
Note that there are no functions $\function{\pi}{U}{\emptyset}$ in $\Set$ unless $U=\emptyset$. We will use the abbreviation $A^{\tensorcheck 0}:=A^{\tensorcheck\emptyset}$. We use the convention that $X^{\wedge 0}:=\capR$ denotes the unit in $(\ModR,\Smash,\capR)$ for any $\capR$-module $X$.
\end{defn}

\begin{defn}\label{defn:circle_product}
Let $A,B,C\in\SymSeq$, and $t,u\geq 0$. The \emph{circle product} (or composition product or substitution product) $A\circ B\in\SymSeq$ is defined objectwise by the coend
\begin{align}
  \label{eq:circle_product_calc}
  (A\circ B)[\mathbf{u}] := A\Smash_\Sigma (B^{\tensorcheck-})[\mathbf{u}]
  &\Iso
  \coprod_{t\geq 0}A[\mathbf{t}]\Smash_{\Sigma_t}
  (B^{\tensorcheck t})[\mathbf{u}].
\end{align}
\end{defn}

\begin{prop}
\label{prop:closed_monoidal_on_symmetric_sequences}
\
\begin{itemize}
\item [(a)] $(\SymSeq,\tensorcheck,1)$ has the structure of a closed symmetric monoidal category with all small limits and colimits. The unit for $\tensorcheck$ denoted ``$1$'' is the symmetric sequence concentrated at $0$ with value $\capR$.
\item [(b)] $(\SymSeq,\circ,I)$ has the structure of a closed monoidal category with all small limits and colimits. The unit for $\circ$ denoted ``$I$'' is the symmetric sequence concentrated at $1$ with value $\capR$. Circle product is not symmetric.
\end{itemize}
\end{prop}

\begin{defn}
\label{defn:hat_construction_embed_at_zero}
Let $Z\in\ModR$. Define $\hat{Z}\in\SymSeq$ to be the symmetric sequence concentrated at $0$ with value $Z$.
\end{defn}

The functor $\function{\hat{-}}{\ModR}{\SymSeq}$ fits into the adjunction
\begin{align*}
\xymatrix{
  \ModR\ar@<0.5ex>[r]^-{\hat{-}} &
  \SymSeq\ar@<0.5ex>[l]^-{\Ev_0}
}
\end{align*}
with left adjoint on top and $\Ev_0$ the \emph{evaluation} functor defined objectwise by $\Ev_0(B):=B[\mathbf{0}]$. Note that $\hat{-}$ embeds $\ModR$ in $\SymSeq$ as the full subcategory of symmetric sequences concentrated at $0$.

\begin{defn}\label{defn:corresponding_functor}
Let $\capO$ be a symmetric sequence and $Z\in\ModR$. The corresponding functor $\functor{\capO}{\ModR}{\ModR}$ is defined objectwise by
$
  \capO(Z):=\capO\circ(Z):=\amalg_{t\geq 0}\capO[\mathbf{t}]
  \Smash_{\Sigma_t}Z^{\wedge t}.
$
\end{defn}

\begin{prop}
Let $\capO,A\in\SymSeq$ and $Z\in\ModR$. There are natural isomorphisms
\begin{align}
\label{eq:circ_product_and_evaluate_at_zero}
  \widehat{\capO(Z)}=
  \widehat{\capO\circ(Z)}\Iso\capO\circ\hat{Z},\quad\quad
  \Ev_0(\capO\circ A)\Iso \capO\circ\bigl(\Ev_0(A)\bigr).
\end{align}
\end{prop}

\begin{proof}
This follows from \eqref{eq:circle_product_calc} and \eqref{eq:tensor_check_calc}.
\end{proof}

\begin{defn}
\label{defn:operad}
An \emph{operad} in $\capR$-modules is a monoid object in $(\SymSeq,\circ,I)$ and a \emph{morphism of operads} is a morphism of monoid objects in $(\SymSeq,\circ,I)$.
\end{defn}

\begin{rem} If $\capO$ is an operad, then the associated functor $\function{\capO}{\ModR}{\ModR}$ inherits the structure of a monad.
\end{rem}

\begin{defn}
\label{defn:algebras_and_modules}
Let $\capO$ be an operad in $\capR$-modules.
\begin{itemize}
\item A \emph{left $\capO$-module} is an object in $(\SymSeq,\circ,I)$ with a left action of $\capO$ and a \emph{morphism of left $\capO$-modules} is a map that respects the left $\capO$-module structure. Denote by $\LtO$ the category of left $\capO$-modules and their morphisms.
\item An \emph{$\capO$-algebra} is an algebra for the monad $\functor{\capO}{\ModR}{\ModR}$ and a \emph{morphism of $\capO$-algebras} is a map in $\ModR$ that respects the $\capO$-algebra structure. Denote by $\AlgO$ the category of $\capO$-algebras and their morphisms.
\end{itemize}
\end{defn}

It follows easily from \eqref{eq:circ_product_and_evaluate_at_zero} that an $\capO$-algebra is the same as an $\capR$-module $Z$ with a left $\capO$-module structure on $\hat{Z}$, and if $Z$ and $Z'$ are $\capO$-algebras, then a morphism of $\capO$-algebras is the same as a map $\function{f}{Z}{Z'}$ in $\ModR$ such that $\function{\hat{f}}{\hat{Z}}{\hat{Z'}}$ is a morphism of left $\capO$-modules. In other words, an algebra over an operad $\capO$ is the same as a left $\capO$-module that is concentrated at $0$, and $\AlgO$ embeds in $\LtO$ as the full subcategory of left $\capO$-modules concentrated at $0$, via the functor $\function{\hat{-}}{\AlgO}{\LtO}$, $Z\longmapsto \hat{Z}$. Define the \emph{evaluation} functor $\function{\Ev_0}{\LtO}{\AlgO}$ objectwise by $\Ev_0(B):=B[\mathbf{0}]$.

\begin{prop}
\label{prop:basic_properties_LTO}
Let $\capO$ be an operad in $\capR$-modules. There are adjunctions
\begin{align}
\label{eq:free_forgetful_adjunction}
\xymatrix{
  \ModR\ar@<0.5ex>[r]^-{\capO\circ(-)} & \AlgO,\ar@<0.5ex>[l]^-{U}
}\quad\quad
\xymatrix{
  \SymSeq\ar@<0.5ex>[r]^-{\capO\circ-} & \LtO,\ar@<0.5ex>[l]^-{U}
}\quad\quad
\xymatrix{
  \AlgO\ar@<0.5ex>[r]^-{\hat{-}} & \LtO,\ar@<0.5ex>[l]^-{\Ev_0}
}
\end{align}
with left adjoints on top and $U$ the forgetful functor. All small colimits exist in $\AlgO$ and $\LtO$, and both reflexive coequalizers and filtered colimits are created by the forgetful functors. All small limits exist in $\AlgO$ and $\LtO$, and are created by the forgetful functors.
\end{prop}

Throughout this paper, we use the following model structures on the categories of $\capO$-algebras and left $\capO$-modules.

\begin{defn}
\label{defn:stable_flat_positive_model_structures}
Let $\capO$ be an operad in $\capR$-modules. The \emph{positive flat stable model structure} on $\AlgO$ (resp. $\LtO$) has as weak equivalences the stable equivalences (resp. objectwise stable equivalences) and  as fibrations the positive flat stable fibrations (resp. objectwise positive flat stable fibrations).
\end{defn}

The model structures in Definition \ref{defn:stable_flat_positive_model_structures} are established in \cite{Harper_Spectra, Harper_Spectra_Corrigendum, Harper_Hess}. For a description of the cofibrations, see \cite[Section 4]{Harper_Spectra} and \cite[Section 7]{Harper_Hess}. For ease of notation, we have followed Schwede \cite{Schwede_book_project} in using the term \emph{flat} (e.g., flat stable model structure) for what is called $S$ (e.g., stable $S$-model structure) in \cite{Hovey_Shipley_Smith, Schwede, Shipley_comm_ring}. For some of the good properties of the flat stable model structure, see \cite[5.3.7 and 5.3.10]{Hovey_Shipley_Smith}.

\section{Homotopical Analysis of Cubical Diagrams}
\label{sec:cubical_diagrams}

In this section we prove the main results of the paper. The following definitions and constructions appear in Goodwillie \cite{Goodwillie_calc2} in the context of spaces, and will also be useful in our context of structured ring spectra.

\begin{defn}[Indexing categories for cubical diagrams]
Let $W$ be a finite set and $\M$ a category.
\begin{itemize}
\item Denote by $\powerset(W)$ the poset of all subsets of $W$, ordered by inclusion $\subset$ of sets. We will often regard $\powerset(W)$ as the category associated to this partial order in the usual way; the objects are the elements of $\powerset(W)$, and there is a morphism $U\rarrow V$ if and only if $U\subset V$.
\item Denote by $\powerset_0(W)\subset\powerset(W)$ the poset of all nonempty subsets of $W$; it is the full subcategory of $\powerset(W)$ containing all objects except the initial object $\emptyset$.
\item Denote by $\powerset_1(W)\subset\powerset(W)$ the poset of all subsets of $W$ not equal to $W$; it is the full subcategory of $\powerset(W)$ containing all objects except the terminal object $W$.
\item A \emph{$W$-cube} $\capX$ in $\M$ is a $\powerset(W)$-shaped diagram $\capX$ in $\M$; in other words, a functor $\function{\capX}{\powerset(W)}{\M}$.
\end{itemize}
\end{defn}

\begin{rem}
If $n=|W|$ and $\capX$ is a $W$-cube in $\M$, we will sometimes refer to $\capX$ simply as an \emph{$n$-cube} in $\M$. In particular, a $0$-cube is an object in $\M$ and a $1$-cube is a morphism in $\M$.
\end{rem}

\begin{defn}[Faces of cubical diagrams]
Let $W$ be a finite set and $\M$ a category. Let $\capX$ be a $W$-cube in $\M$ and consider any subsets $U\subset V\subset W$. Denote by $\partial_U^V\capX$ the $(V-U)$-cube defined objectwise by
\begin{align*}
  T\mapsto(\partial_U^V\capX)_T:=\capX_{T\cup U},\quad\quad T\subset V-U.
\end{align*}
In other words, $\partial_U^V\capX$ is the $(V-U)$-cube formed by all maps in $\capX$ between $\capX_U$ and $\capX_V$. We say that $\partial_U^V\capX$ is a \emph{face} of $\capX$ of \emph{dimension} $|V-U|$.
\end{defn}

\begin{defn}
\label{defn:cofibration_cubes_etc}
Let $\capO$ be an operad in $\capR$-modules and $W$ a finite set. Let $\capX$ be a $W$-cube in $\AlgO$ (resp. $\LtO$) or $\ModR$ (resp. $\SymSeq$) and $k\in\ZZ$.
\begin{itemize}
\item $\capX$ is a \emph{cofibration cube} if the map
$\colim_{\powerset_1(V)}\capX\rarrow\colim_{\powerset(V)}\capX\Iso\capX_V$ is a cofibration for each $V\subset W$; in particular, each $\capX_V$ is cofibrant.
\item $\capX$ is \emph{$k$-cocartesian} if the map
$\hocolim_{\powerset_1(W)}\capX\rarrow\hocolim_{\powerset(W)}\capX\wequiv\capX_W$ is $k$-connected.
\item $\capX$ is \emph{$\infty$-cocartesian} if the map
$\hocolim_{\powerset_1(W)}\capX\rarrow\hocolim_{\powerset(W)}\capX\wequiv\capX_W$ is a weak equivalence.
\item $\capX$ is \emph{strongly $\infty$-cocartesian} if each face of dimension $\geq 2$ is $\infty$-cocartesian.
\item $\capX$ is a \emph{pushout cube} if the map
$\colim_{\powerset_1(V)}\capX\rarrow\colim_{\powerset(V)}\capX\Iso\capX_V$  is an isomorphism for each $V\subset W$ with $|V| \geq 2$; i.e., if it is built by colimits in the usual way out of the maps $\capX_\emptyset\rarrow\capX_V$, $V\subset W$, $|V|=1$.
\end{itemize}
\end{defn}

These definitions and constructions dualize as follows. Note that when looking for the appropriate dual construction, it is useful to observe that $\capX=\partial_\emptyset^V\capX$ when restricted to $\powerset(V)$; for instance, $\colim_{\powerset_1(V)}\capX=\colim_{\powerset_1(V)}\partial_\emptyset^V\capX$.

\begin{defn}
\label{defn:fibration_cubes_etc}
Let $\capO$ be an operad in $\capR$-modules and $W$ a finite set. Let $\capX$ be a $W$-cube in $\AlgO$ (resp. $\LtO$) or $\ModR$ (resp. $\SymSeq$) and $k\in\ZZ$.
\begin{itemize}
\item $\capX$ is a \emph{fibration cube} if the map $\capX_V\Iso\lim_{\powerset(W-V)}\partial_V^W\capX\rarrow\lim_{\powerset_0(W-V)}\partial_V^W\capX$ is a fibration for each $V\subset W$; in particular, each $\capX_V$ is fibrant.
\item $\capX$ is \emph{$k$-cartesian} if the map $\capX_\emptyset\wequiv\holim_{\powerset(W)}\capX\rarrow\holim_{\powerset_0(W)}\capX$ is $k$-connected.
\item $\capX$ is \emph{$\infty$-cartesian} if the map $\capX_\emptyset\wequiv\holim_{\powerset(W)}\capX\rarrow\holim_{\powerset_0(W)}\capX$ is a weak equivalence.
\item $\capX$ is \emph{strongly $\infty$-cartesian} if each face of dimension $\geq 2$ is $\infty$-cartesian.
\item $\capX$ is a \emph{pullback cube} if the map $\capX_V\Iso\lim_{\powerset(W-V)}\partial_V^W\capX\rarrow\lim_{\powerset_0(W-V)}\partial_V^W\capX$ is an isomorphism for each $V\subset W$ with $|W-V| \geq 2$; i.e., if it is built by limits in the usual way out of the maps $\capX_V\rarrow\capX_W$, $V\subset W$, $|W-V|=1$.
\end{itemize}
\end{defn}

\begin{rem}
It is important to note that every $1$-cube in $\AlgO$, $\LtO$, $\ModR$, or $\SymSeq$ is strongly $\infty$-cocartesian (resp. strongly $\infty$-cartesian), since there are no faces of dimension $\geq 2$, but only the $1$-cubes that are weak equivalences are $\infty$-cocartesian (resp. $\infty$-cartesian).
\end{rem}

\begin{prop}
\label{prop:connectivity_estimates_for_composition_of_maps}
Let $k\in\ZZ$. Consider any maps $X\rarrow Y\rarrow Z$ in $\AlgO$ (resp. $\LtO$) or $\ModR$ (resp. $\SymSeq$).
\begin{itemize}
\item[(a)] If $X\rarrow Y$ and $Y\rarrow Z$ are $k$-connected, then $X\rarrow Z$ is $k$-connected.
\item[(b)] If $X\rarrow Y$ is $(k-1)$-connected and $X\rarrow Z$ is $k$-connected, then $Y\rarrow Z$ is $k$-connected.
\item[(c)] If $X\rarrow Z$ is $k$-connected and $Y\rarrow Z$ is $(k+1)$-connected, then $X\rarrow Y$ is $k$-connected.

\end{itemize}
\end{prop}

\begin{proof}
This is because the forgetful functor to symmetric spectra creates $k$-connected maps.
\end{proof}

Versions of the following connectivity estimates are proved in Goodwillie \cite[1.6--1.8]{Goodwillie_calc2} in the context of spaces, and exactly the same arguments give a proof of Propositions \ref{prop:map_of_cubical_diagrams} and \ref{prop:composed_map_of_cubical_diagrams} below in the context of structured ring spectra; this is an exercise left to the reader.

\begin{prop}
\label{prop:map_of_cubical_diagrams}
Let $W$ be a finite set and $k\in\ZZ$. Consider any map $\capX\rarrow\capY$ of $W$-cubes in $\AlgO$ (resp. $\LtO$) or $\ModR$ (resp. $\SymSeq$).
\begin{itemize}
\item[(a)] If $\capX\rarrow\capY$ and $\capX$ are $k$-cocartesian, then $\capY$ is $k$-cocartesian.
\item[(b)] If $\capX$ is $(k-1)$-cocartesian and $\capY$ is $k$-cocartesian, then $\capX\rarrow\capY$ is $k$-cocartesian.
\item[(c)] If $\capX\rarrow\capY$ and $\capY$ are $k$-cartesian, then $\capX$ is $k$-cartesian.
\item[(d)] If $\capX$ is $k$-cartesian and $\capY$ is $(k+1)$-cartesian, then $\capX\rarrow\capY$ is $k$-cartesian.

\end{itemize}
\end{prop}

\begin{prop}
\label{prop:composed_map_of_cubical_diagrams}
Let $W$ be a finite set and $k\in\ZZ$. Consider any map $\capX\rarrow\capY\rarrow\capZ$ of $W$-cubes in $\AlgO$ (resp. $\LtO$) or $\ModR$ (resp. $\SymSeq$).
\begin{itemize}
\item[(a)] If $\capX\rarrow\capY$ and $\capY\rarrow\capZ$ are $k$-cocartesian, then $\capX\rarrow\capZ$ is $k$-cocartesian.
\item[(b)] If $\capX\rarrow\capY$ is $(k-1)$-cocartesian and $\capX\rarrow\capZ$ is $k$-cocartesian, then $\capY\rarrow\capZ$ is $k$-cocartesian.
\item[(c)] If $\capX\rarrow\capY$ and $\capY\rarrow\capZ$ are $k$-cartesian, then $\capX\rarrow\capZ$ is $k$-cartesian.
\item[(d)] If $\capX\rarrow\capZ$ is $k$-cartesian and $\capY\rarrow\capZ$ is $(k+1)$-cartesian, then $\capX\rarrow\capY$ is $k$-cartesian.
\end{itemize}
\end{prop}

The following results depend on the fact that the model structures on $\ModR$ and $\SymSeq$ are stable, so that fibration and cofibration sequences coincide. Note that these do not hold, in general, for $\AlgO$ and $\LtO$.

\begin{prop}
\label{prop:comparing_cocartesian_and_cartesian_estimates_in_ModR}
Let $W$ be a finite set and $k\in\ZZ$. Let $\capX$ be a $W$-cube in $\ModR$ (resp. $\SymSeq$).
\begin{itemize}
\item[(a)] $\capX$ is $k$-cocartesian if and only if $\capX$ is $(k-|W|+1)$-cartesian.
\item[(b)] $\capX$ is $k$-cartesian if and only if $\capX$ is $(k+|W|-1)$-cocartesian.
\end{itemize}
\end{prop}

\begin{proof}
This is because the total homotopy cofiber of $\capX$ (see Goodwillie \cite[1.4]{Goodwillie_calc2}) is weakly equivalent to the $|W|$-th suspension, usually denoted $\Sigma^{|W|}$, of the total homotopy fiber of $\capX$ (see \cite[1.1a]{Goodwillie_calc2}).
\end{proof}

\subsection{Proof of higher homotopy excision for $\AlgO$ and $\LtO$}

The purpose of this section is to prove Theorem~\ref{thm:higher_homotopy_excision}. At the heart of our proof is a homotopical analysis of the construction $\capO_A$ described in Proposition~\ref{prop:coproduct_modules}. We deduce Theorem~\ref{thm:higher_homotopy_excision} from a more general result about the effect of the construction $A \mapsto \capO_A$ on strongly $\infty$-cocartesian cubes. 

\begin{defn}\label{def:symmetric_array}
Consider symmetric sequences in $\ModR$. A \emph{symmetric array} in $\ModR$ is a symmetric sequence in $\SymSeq$; i.e., a functor $\functor{A}{\Sigma^\op}{\SymSeq}$. Denote by $\SymArray:=\SymSeq^{\Sigma^\op}$ the category of symmetric arrays in $\ModR$ and their natural transformations.
\end{defn}

A first step in analyzing the pushouts in \eqref{eq:small_arg_pushout_modules} below is an analysis of certain coproducts. The following proposition is motivated by Goerss-Hopkins \cite[Section 2.3]{Goerss_Hopkins_moduli} and Mandell \cite[Section 13]{Mandell}; a proof is given in \cite[4.7]{Harper_Spectra}. The $\capO_A$ construction that arises here is crucial to our arguments.

\begin{prop}
\label{prop:coproduct_modules}
Let $\capO$ be an operad in $\ModR$, $A\in\AlgO$ (resp. $A\in\LtO$), and $Y\in\ModR$ (resp. $Y\in\SymSeq$). Consider any coproduct in $\AlgO$ (resp. $\LtO$) of the form $A\amalg\capO\circ(Y)$ (resp. $A\amalg(\capO\circ Y)$). There exists a symmetric sequence $\capO_A$ (resp. symmetric array $\capO_A$) and natural isomorphisms
\begin{align*}
  A\amalg\capO\circ(Y) \Iso
  \coprod\limits_{q\geq 0}\capO_A[\mathbf{q}]
  \Smash_{\Sigma_q}Y^{\wedge q}\quad
  \Bigl(\text{resp.}\quad
  A\amalg(\capO\circ Y) \Iso
  \coprod\limits_{q\geq 0}\capO_A[\mathbf{q}]
  \tensorcheck_{\Sigma_q}Y^{\tensorcheck q}
    \Bigr)
\end{align*}
in the underlying category $\ModR$ (resp. $\SymSeq$). For any $q\geq 0$, then $\capO_A[\mathbf{q}]$ is naturally isomorphic to a colimit of the form
\begin{align*}
  \capO_A[\mathbf{q}]&\Iso
  \colim\biggl(
  \xymatrix{
    \coprod\limits_{p\geq 0}\capO[\mathbf{p}\boldsymbol{+}\mathbf{q}]
    \Smash_{\Sigma_p}A^{\wedge p} &
    \coprod\limits_{p\geq 0}\capO[\mathbf{p}\boldsymbol{+}\mathbf{q}]
    \Smash_{\Sigma_p}(\capO\circ (A))^{\wedge p}\ar@<-0.5ex>[l]^-{d_1}
    \ar@<-1.5ex>[l]_-{d_0}
  }
  \biggl),\\
  \text{resp.}\quad
  \capO_A[\mathbf{q}]&\Iso
  \colim\biggl(
  \xymatrix{
    \coprod\limits_{p\geq 0}\capO[\mathbf{p}\boldsymbol{+}\mathbf{q}]
    \Smash_{\Sigma_p}A^{\tensorcheck p} &
    \coprod\limits_{p\geq 0}\capO[\mathbf{p}\boldsymbol{+}\mathbf{q}]
    \Smash_{\Sigma_p}(\capO\circ A)^{\tensorcheck p}\ar@<-0.5ex>[l]^-{d_1}
    \ar@<-1.5ex>[l]_-{d_0}
  }
  \biggl),
\end{align*}
in $\ModR^{\Sigma_q^\op}$ (resp. $\SymSeq^{\Sigma_q^\op}$), with $d_0$ induced by operad multiplication and $d_1$ induced by the left $\capO$-action map $\function{m}{\capO\circ (A)}{A}$ (resp. $\function{m}{\capO\circ A}{A}$).
\end{prop}

\begin{rem}
Other possible notations for $\capO_A$ include $\U_\capO(A)$ or $\U(A)$; these are closer to the notation used in \cite{Elmendorf_Mandell, Mandell} and are not to be confused with the forgetful functors. It is interesting to note---although we will not use it in this paper---that in the context of $\capO$-algebras the symmetric sequence $\capO_A$ has the structure of an operad; it parametrizes $\capO$-algebras under $A$ and is sometimes called the enveloping operad for $A$. It is for this purpose that the $\capO_A$ construction appears in \cite{Fresse_lie_theory}.
\end{rem}

Recall from \cite{Harper_Hess} the following proposition.

\begin{prop}
\label{prop:OA_commutes_with_certain_colimits}
Let $\capO$ be an operad in $\ModR$ and let $q\geq 0$. Then the functor
$
  \function{\capO_{(-)}[\mathbf{q}]}{\AlgO}{\ModR^{\Sigma_q^\op}}
$ (resp.
$
  \function{\capO_{(-)}[\mathbf{q}]}{\LtO}{\SymSeq^{\Sigma_q^\op}}
$) preserves reflexive coequalizers and filtered colimits.
\end{prop}

\begin{rem}
Reflexive coequalizers and filtered colimits are the two main examples of sifted colimits. Although we will not need this generalization here, the functor $\capO_{(-)}[\mathbf{q}]$ commutes with all sifted colimits.
 \end{rem}

\begin{defn}\label{def:filtration_setup_modules}
Let $\function{i}{X}{Y}$ be a morphism in $\ModR$ (resp. $\SymSeq$) and $t\geq 1$. Define $Q_0^t:=X^{\wedge t}$ (resp. $Q_0^t:=X^{\tensorcheck t}$) and $Q_t^t:=Y^{\wedge t}$ (resp. $Q_t^t:=Y^{\tensorcheck t}$). For $0<q<t$ define $Q_q^t$ inductively by the left-hand (resp. right-hand) pushout diagrams
\begin{align*}
\xymatrix{
  \Sigma_t\cdot_{\Sigma_{t-q}\times\Sigma_{q}}X^{\wedge(t-q)}
  \Smash Q_{q-1}^q\ar[d]^{i_*}\ar[r]^-{\pr_*} & Q_{q-1}^t\ar[d]\\
  \Sigma_t\cdot_{\Sigma_{t-q}\times\Sigma_{q}}X^{\wedge(t-q)}
  \Smash Y^{\wedge q}\ar[r] & Q_q^t
}\quad
\xymatrix{
  \Sigma_t\cdot_{\Sigma_{t-q}\times\Sigma_{q}}X^{\tensorcheck(t-q)}
  \tensorcheck Q_{q-1}^q\ar[d]^{i_*}\ar[r]^-{\pr_*} & Q_{q-1}^t\ar[d]\\
  \Sigma_t\cdot_{\Sigma_{t-q}\times\Sigma_{q}}X^{\tensorcheck(t-q)}
  \tensorcheck Y^{\tensorcheck q}\ar[r] & Q_q^t
}
\end{align*}
in $\ModR^{\Sigma_t}$ (resp. $\SymSeq^{\Sigma_t}$). We sometimes denote $Q_q^t$ by $Q_q^t(i)$ to emphasize in the notation the map $\function{i}{X}{Y}$. The maps $\pr_*$ and $i_*$ are the obvious maps induced by $i$ and the appropriate projection maps; see, for instance, \cite[4.15]{Harper_Spectra} for a useful elaboration of this construction in low dimensions.
\end{defn}

The following filtrations of Elmendorf-Mandell \cite{Elmendorf_Mandell} provide one of the key technical tools needed for establishing the main results in this paper; their construction is motivated by \cite[Section 11]{Elmendorf_Mandell} where they appear for the case $r=0$ in the context of simplicial multifunctors of symmetric spectra. The refinement to $r\geq 0$ is motivated by comparing \cite[Section 11]{Elmendorf_Mandell} with Mandell \cite[13.7]{Mandell}. These filtrations have been exploited as a kind of secret weapon in \cite{Goerss_Hopkins_moduli, Harper_Spectra, Harper_Bar, Harper_Modules, Harper_Hess}; for other approaches to these types of filtrations compare \cite{Fresse_modules, Schwede_Shipley}. A proof of the Elmendorf-Mandell filtrations, in their refined form as described by the following proposition, is given in \cite{Harper_Hess}.

\begin{prop}
\label{prop:filtering_OA}
Let $\capO$ be an operad in $\ModR$, $A\in\AlgO$ (resp. $A\in\LtO$), and $\function{i}{X}{Y}$ in $\ModR$ (resp. $\SymSeq$). Consider any pushout diagram in $\AlgO$ (resp. $\LtO$) of the form
\begin{align}
\label{eq:small_arg_pushout_modules}
\xymatrix{
  \capO\circ (X)\ar[r]^-{f}\ar[d]^{\id\circ (i)} & A\ar[d]^{j}\\
  \capO\circ (Y)\ar[r] & B
}\quad\quad\text{resp.}\quad
\xymatrix{
  \capO\circ X\ar[r]^-{f}\ar[d]^{\id\circ i} & A\ar[d]^{j}\\
  \capO\circ Y\ar[r] & B
}
\end{align}
For each $r\geq 0$, $\capO_B[\mathbf{r}]$ is naturally isomorphic to a filtered colimit of the form
\begin{align}
\label{eq:filtered_colimit_modules_refined}
  \capO_B[\mathbf{r}]\Iso
  \colim\Bigl(
  \xymatrix{
    \capO_A^0[\mathbf{r}]\ar[r]^{j_1} &
    \capO_A^1[\mathbf{r}]\ar[r]^{j_2} &
    \capO_A^2[\mathbf{r}]\ar[r]^{j_3} & \dotsb
  }
  \Bigr)
\end{align}
in $\ModR^{\Sigma_r^\op}$ (resp. $\SymSeq^{\Sigma_r^\op}$), with $\capO_A^0[\mathbf{r}]:=\capO_A[\mathbf{r}]$ and $\capO_A^t[\mathbf{r}]$ defined inductively by pushout diagrams in $\ModR^{\Sigma_r^\op}$ (resp. $\SymSeq^{\Sigma_r^\op}$) of the form
\begin{align}
\label{eq:good_filtration_modules_refined}
\xymatrix{
  \capO_A[\mathbf{t}\boldsymbol{+}\mathbf{r}]\Smash_{\Sigma_t}Q_{t-1}^t\ar[d]^{\id\wedge_{\Sigma_t}i_*}
  \ar[r]^-{f_*} & \capO_A^{t-1}[\mathbf{r}]\ar[d]^{j_t}\\
  \capO_A[\mathbf{t}\boldsymbol{+}\mathbf{r}]\Smash_{\Sigma_t}Y^{\wedge t}\ar[r]^-{\xi_t} & \capO_A^t[\mathbf{r}]
}\quad\quad
\text{resp.}\quad
\xymatrix{
  \capO_A[\mathbf{t}\boldsymbol{+}\mathbf{r}]\tensorcheck_{\Sigma_t}Q_{t-1}^t\ar[d]^{\id\tensorcheck_{\Sigma_t}i_*}
  \ar[r]^-{f_*} & \capO_A^{t-1}[\mathbf{r}]\ar[d]^{j_t}\\
  \capO_A[\mathbf{t}\boldsymbol{+}\mathbf{r}]\tensorcheck_{\Sigma_t}Y^{\tensorcheck t}\ar[r]^-{\xi_t} & \capO_A^t[\mathbf{r}]
}
\end{align}
\end{prop}

\begin{rem}
It is important to note (see \cite{Harper_Hess}) that for $r=0$ the filtration \eqref{eq:filtered_colimit_modules_refined} specializes to a filtered colimit of the pushout in \eqref{eq:small_arg_pushout_modules} of the form
\begin{align}
\label{eq:good_filtration_modules}
  B\Iso
  \capO_B[\mathbf{0}]\Iso
  \colim\bigl(
  \xymatrix@1{
    A_0\ar[r]^{j_1} & A_1\ar[r]^{j_2} & A_2\ar[r]^{j_3} & \dotsb
  }
  \bigr)
\end{align}
in the underlying category $\ModR$ (resp. $\SymSeq$), with $A_0:=\capO_A[\mathbf{0}]\Iso A$ and $A_t:=\capO_A^t[\mathbf{0}]$.
\end{rem}

\begin{prop}\label{prop:pushout_of_n_connected_cofibration}
Let $n\geq -1$. If the map $\function{i}{X}{Y}$ in Proposition \ref{prop:filtering_OA} is an $n$-connected generating cofibration or generating acyclic cofibration in $\ModR$ (resp. $\SymSeq$) with the positive flat stable model structure, and $\capR,\capO_A$ are $(-1)$-connected, then each map $j_t$ in \eqref{eq:good_filtration_modules} and \eqref{eq:good_filtration_modules_refined} is an $n$-connected monomorphism. In particular, the map $j$ in \eqref{eq:small_arg_pushout_modules} is an $n$-connected monomorphism in the underlying category $\ModR$ (resp. $\SymSeq$).
\end{prop}

\begin{proof}
It suffices to consider the case of left $\capO$-modules. The generating cofibrations (see \cite[4.2(b)]{Harper_Spectra} or \cite[7.10(b)]{Harper_Hess}) and acyclic cofibrations in $\SymSeq$ have cofibrant domains. Hence by \cite[4.28*]{Harper_Spectra_Corrigendum}, each $j_t$ in \eqref{eq:good_filtration_modules} is a monomorphism. We know that $A_t/A_{t-1}\Iso \capO_A[\mathbf{t}]\tensorcheck_{\Sigma_t}(Y/X)^{\tensorcheck t}$ and $*\rarrow Y/X$ is an $n$-connected cofibration in $\SymSeq$. It follows from \cite[4.40]{Harper_Hess} that each $j_t$ in \eqref{eq:good_filtration_modules} is $n$-connected. The case for each map $j_t$ in \eqref{eq:good_filtration_modules_refined} is similar. Here, it is important to note that since $\Smash$ denotes the smash product of $\capR$-modules (see \cite[7.4]{Harper_Hess}), it is given by $\Smash=(\tensor_S)_\capR$ first tensoring over the sphere spectrum $S$ and then further dividing out by the $\capR$-action; in particular, we have used the assumption that $\capR$ is $(-1)$-connected to ensure that the indicated smash powers behave as desired with respect to connectivity.
\end{proof}

The following proposition is closely related to Dugger-Shipley \cite[A.3, A.4]{Dugger_Shipley}.

\begin{prop}
\label{prop:factorization_into_n_connected_cofibration}
Let $\capO$ be an operad in $\capR$-modules and $n\geq -1$. If $\function{f}{A}{C}$ is an $n$-connected map in $\AlgO$ (resp. $\LtO$) and $\capR,\capO_A$ are $(-1)$-connected, then $f$ factors in $\AlgO$ (resp. $\LtO$) as
\begin{align}
\label{eq:factoring_n_connected_maps}
\xymatrix{
  A\ar[r]^-{j} & B\ar[r]^-{p} & C
}
\end{align}
a nice $n$-connected cofibration followed by an acyclic fibration. Here, ``nice'' means that $j$ is a (possibly transfinite) composition of pushouts of $n$-connected generating cofibrations and generating acyclic cofibrations in $\AlgO$ (resp. $\LtO$). This factorization is functorial in all such $f$.
\end{prop}

\begin{proof}[Proof of Proposition \ref{prop:factorization_into_n_connected_cofibration}]
It suffices to consider the case of left $\capO$-modules. Let $\function{i}{X}{Y}$ be an $n$-connected generating cofibration or generating acyclic cofibration in $\SymSeq$ with the positive flat stable model structure, and consider the pushout diagram
\begin{align}
\label{eq:gluing_on_cells}
\xymatrix{
  \capO\circ X\ar[r]\ar[d] & Z_0\ar[d]^{i_0}\\
  \capO\circ Y\ar[r] & Z_1
}
\end{align}
in $\LtO$. Assume $\capO_{Z_0}$ is $(-1)$-connected; let's verify that $\capO_{Z_1}$ is $(-1)$-connected and $i_0$ is an $n$-connected monomorphism in $\SymSeq$. Let $A:=Z_0$. By Proposition \ref{prop:filtering_OA}, we know $\capO_{Z_1}[\mathbf{r}]$ is naturally isomorphic to a filtered colimit of the form
\begin{align*}
  \capO_{Z_1}[\mathbf{r}]&\Iso
  \colim\bigl(
  \xymatrix{
    \capO_A^0[\mathbf{r}]\ar[r]^{j_1} &
    \capO_A^1[\mathbf{r}]\ar[r]^{j_2} &
    \capO_A^2[\mathbf{r}]\ar[r]^{j_3} & \dotsb
  }
  \bigr)
\end{align*}
in $\SymSeq^{\Sigma_r^\op}$, and Proposition \ref{prop:pushout_of_n_connected_cofibration} verifies that each $j_t$ is an $n$-connected monomorphism. Since $\capO_A^0=\capO_{Z_0}$ is $(-1)$-connected by assumption, it follows that $\capO_{Z_1}$ is $(-1)$-connected, and taking $r=0$ (or using Proposition \ref{prop:pushout_of_n_connected_cofibration} again) finishes the argument that $i_0$ is an $n$-connected monomorphism in $\SymSeq$.

Consider a sequence
$
\xymatrix@1{
  Z_0\ar[r] & Z_1\ar[r] & Z_2\ar[r] & \dotsb
}
$
of pushouts of maps as in \eqref{eq:gluing_on_cells}, and let $Z_\infty:=\colim_k Z_k$. Consider the naturally occurring map $Z_0\rarrow Z_\infty$, and assume $\capO_{Z_0}$ is $(-1)$-connected. By the argument above we know that
\begin{align*}
\xymatrix{
  \capO_{Z_0}[\mathbf{r}]\ar[r] &
  \capO_{Z_1}[\mathbf{r}]\ar[r] &
  \capO_{Z_2}[\mathbf{r}]\ar[r] & \dotsb
}
\end{align*}
is a sequence of $n$-connected monomorphisms, hence $\capO_{Z_\infty}$ is $(-1)$-connected, and taking $r=0$ verifies that $Z_0\rarrow Z_\infty$ is an $n$-connected monomorphism in $\SymSeq$.

The small object argument (see \cite[7.12]{Dwyer_Spalinski} for a useful introduction) produces a factorization \eqref{eq:factoring_n_connected_maps} of $f$ such that $p$, and hence its fiber $F\rightarrow *$, has the right lifting property with respect to the $n$-connected generating cofibrations and generating acyclic cofibrations in $\LtO$, and $j$ is a (possibly transfinite) composition of pushouts of maps as in \eqref{eq:gluing_on_cells}, starting with $Z_0=A$. It follows from the latter lifting property that both $p$ and its pullback $F\rightarrow *$ are fibrations in $\LtO$; in particular, $F[\mathbf{u}]$ is stably fibrant and hence is an $\Omega$-spectrum \cite[1.4]{Hovey_Shipley_Smith} for each $u\geq 0$. By the argument above, it follows that $j$ is $n$-connected. Since $f$ is $n$-connected by assumption, it follows that $p$ is $n$-connected and therefore $F\rightarrow *$ is $n$-connected.  Since $F\rightarrow *$, furthermore, has the right lifting property with respect to the $n$-connected generating cofibrations, it follows that the homotopy groups $\pi_kF[\mathbf{u}]=0$ are trivial for each $k\geq n$, $u\geq 0$, and hence $F\rightarrow *$ is a weak equivalence; therefore $p$ is a weak equivalence which completes the proof.
\end{proof}

\begin{rem}
To keep the statement of Proposition \ref{prop:factorization_into_n_connected_cofibration} as simple and non-technical as possible, we have been conservative in our choice for the set of maps used in the small object argument. In other words, running the small object argument with the set of $n$-connected generating cofibrations and generating acyclic cofibrations in $\AlgO$ (resp. $\LtO$) is sufficient for our purposes and makes for an attractive and simple statement, but one can obtain the desired factorizations using a smaller set of maps; this is an exercise left to the reader.
\end{rem}

\begin{prop}
\label{prop:OA_construction_on_n_connected_cofibrations}
Let $\capO$ be an operad in $\capR$-modules and $n\geq -1$.  Let $\function{j}{A}{B}$ be a cofibration in $\AlgO$ (resp. $\LtO$). Assume that $\capR,\capO_A$ are $(-1)$-connected. If $j$ is $n$-connected, then $\capO_A[\mathbf{r}]\rarrow\capO_B[\mathbf{r}]$ is an $n$-connected monomorphism for each $r\geq 0$.
\end{prop}

\begin{proof}
It suffices to consider the case of left $\capO$-modules. Proceed exactly as in the proof of Proposition \ref{prop:factorization_into_n_connected_cofibration}, and consider the pushout diagram \eqref{eq:gluing_on_cells}. Assume $\capO_{Z_0}$ is $(-1)$-connected; let's verify that $\capO_{Z_0}[\mathbf{r}]\rarrow\capO_{Z_1}[\mathbf{r}]$ is an $n$-connected monomorphism for each $r\geq 0$. This follows by arguing exactly as in the proof of Proposition \ref{prop:factorization_into_n_connected_cofibration}. Consider a sequence
\begin{align*}
\xymatrix@1{
  Z_0\ar[r] & Z_1\ar[r] & Z_2\ar[r] & \dotsb
}
\end{align*}
of pushouts of maps as in \eqref{eq:gluing_on_cells}, and let $Z_\infty:=\colim_k Z_k$. Consider the naturally occurring map $Z_0\rarrow Z_\infty$, and assume $\capO_{Z_0}$ is $(-1)$-connected. By the argument above we know that
\begin{align*}
\xymatrix{
  \capO_{Z_0}[\mathbf{r}]\ar[r] &
  \capO_{Z_1}[\mathbf{r}]\ar[r] &
  \capO_{Z_2}[\mathbf{r}]\ar[r] & \dotsb
}
\end{align*}
is a sequence of $n$-connected monomorphisms, hence $\capO_{Z_0}[\mathbf{r}]\rarrow \capO_{Z_\infty}[\mathbf{r}]$ is an $n$-connected monomorphism. Noting that every $n$-connected cofibration of the form $A\rarrow B$ in $\LtO$ is a retract of a (possibly transfinite) composition of pushouts of maps as in \eqref{eq:gluing_on_cells}, starting with $Z_0=A$, finishes the proof.
\end{proof}

\begin{prop}
Let $\capO$ be an operad in $\capR$-modules. If $A$ is a cofibrant $\capO$-algebra (resp. left $\capO$-module) and $\capR,\capO,A$ are $(-1)$-connected, then $\capO_A$ is $(-1)$-connected.
\end{prop}

\begin{proof}
It suffices to consider the case of left $\capO$-modules. This follows from Proposition \ref{prop:OA_construction_on_n_connected_cofibrations} by considering the map $\capO\circ\emptyset\rarrow A$ in $\LtO$, together with the natural isomorphisms $\capO\circ\emptyset\Iso\widehat{\capO[\mathbf{0}]}$ and $\capO_{\capO\circ\emptyset}[\mathbf{r}]\Iso\widehat{\capO[\mathbf{r}]}$ for each $r\geq 0$ (see \cite[5.31]{Harper_Hess}).
\end{proof}

\begin{rem}
\label{rem:map_of_n_cubes}
Any $3$-cube $\mathcal{X}$ of $\capO$-algebras (resp. left $\capO$-modules) may be regarded as a map of $2$-cubes $A\rarrow B$ with $A=\partial_\emptyset^{\{1,2\}}\capX$ (the top face of $\capX$) and $B=\partial_{\{3\}}^{\{1,2,3\}}\capX$ (the bottom face of $\capX$) as follows:
\begin{align*}
\xymatrix@!0{
\capX_\emptyset\ar[rr]\ar[dd]\ar[dr] &&
\capX_{\{1\}}\ar[dr]\ar'[d][dd]\\
&\capX_{\{2\}}\ar[rr]\ar[dd] &&
\capX_{\{1,2\}}\ar[dd]\\
\capX_{\{3\}}\ar[dr]\ar'[r][rr] &&
\capX_{\{1,3\}}\ar[dr]\\
&\capX_{\{2,3\}}\ar[rr] &&
\capX_{\{1,2,3\}}
}
\quad\quad
\xymatrix@!0{
A_\emptyset\ar[rr]\ar[dd]\ar[dr] &&
A_{\{1\}}\ar[dr]\ar'[d][dd]\\
&A_{\{2\}}\ar[rr]\ar[dd] &&
A_{\{1,2\}}\ar[dd]\\
B_\emptyset\ar[dr]\ar'[r][rr] &&
B_{\{1\}}\ar[dr]\\
&B_{\{2\}}\ar[rr] &&
B_{\{1,2\}}
}
\end{align*}
More generally, we may regard an $(n+1)$-cube of $\capO$-algebras (resp. left $\capO$-modules) as a map of $n$-cubes $A\rarrow B$ with $A=\partial_\emptyset^{\{1,\dots,n\}}\capX$ and $B=\partial_{\{n+1\}}^{\{1,\dots,n+1\}}\capX$, for each $n\geq 0$. In particular, the map $\capX_\emptyset\rarrow\capX_{\{n+1\}}$ in $\capX$ is the map $A_\emptyset\rarrow B_\emptyset$ in $A\rarrow B$.
\end{rem}

We now prove the following result of which Theorem~\ref{thm:higher_homotopy_excision} is the special case $r = 0$. 

\begin{thm}[Homotopical analysis of $\capO_\capX$ for a pushout cofibration cube $\capX$]
\label{thm:pushout_cofibration_cube_homotopical_analysis}
Let $\capO$ be an operad in $\capR$-modules and $n\geq 1$. Let $\capX$ be a pushout $(\mathbf{n}\boldsymbol{+}\mathbf{1})$-cube of $\capO$-algebras (resp. left $\capO$-modules) regarded as a map of pushout $\mathbf{n}$-cubes $A\rarrow B$ as in Remark \ref{rem:map_of_n_cubes}. Assume that $\capR,\capO_{A_\emptyset}$ are $(-1)$-connected. Let $k_1,\dots,k_{n+1}\geq -1$. Assume that each $A_\emptyset\rarrow A_{\{i\}}$ and $A_\emptyset\rarrow B_\emptyset$ are cofibrations between cofibrant objects in $\AlgO$ (resp. $\LtO$) $(1\leq i\leq n)$. Consider the associated left-hand diagram of the form
\begin{align}
\label{eq:colim_punctured_cube_associated_diagrams}
\xymatrix{
  \colim\limits_{\ \ \powerset_1(\mathbf{n})}A\ar[d]\ar[r] & \tilde{A}\ar[d]\\
  \colim\limits_{\ \ \powerset_1(\mathbf{n})}B\ar[r] & \tilde{B}
}\quad\quad
\xymatrix{
  \colim\limits_{\ \ \powerset_1(\mathbf{n})}\capO_A[\mathbf{r}]\ar[d]\ar[r] &
  \capO_{\tilde{A}}[\mathbf{r}]\ar[d]\\
  \colim\limits_{\ \ \powerset_1(\mathbf{n})}\capO_B[\mathbf{r}]\ar[r] &
  \capO_{\tilde{B}}[\mathbf{r}]
}
\end{align}
in the underlying category $\ModR$ (resp. $\SymSeq$), and more generally, the associated right-hand diagrams $(r\geq 0)$ in $\ModR^{\Sigma_r^\op}$ (resp. $\SymSeq^{\Sigma_r^\op}$). If each $A_\emptyset\rarrow A_{\{i\}}$ is $k_i$-connected $(1\leq i\leq n)$ and $A_\emptyset\rarrow B_\emptyset$ is $k_{n+1}$-connected, then the diagrams \eqref{eq:colim_punctured_cube_associated_diagrams} are $(k_1+\dots + k_{n+1}+n)$-cocartesian; here, $\tilde{A}:=A_\mathbf{n}$ and $\tilde{B}:=B_{\mathbf{n}}$.
\end{thm}

\begin{rem}
In other words, this theorem shows that the $(\mathbf{n}\boldsymbol{+}\mathbf{1})$-cube $\capO_\capX[\mathbf{r}]$ is $(k_1+\dots + k_{n+1}+n)$-cocartesian for each $r\geq 0$, or equivalently, it shows that the $(\mathbf{n}\boldsymbol{+}\mathbf{1})$-cube $\capO_\capX$ is $(k_1+\dots + k_{n+1}+n)$-cocartesian. The left-hand diagram in (\ref{eq:colim_punctured_cube_associated_diagrams}) is the case $r = 0$ and the result here is precisely that needed for Theorem~\ref{thm:higher_homotopy_excision}.
\end{rem}

\begin{proof}
It suffices to consider the case of left $\capO$-modules. The argument is by induction on $n$. It is convenient to start the induction at $n=0$ in which case the diagrams in \eqref{eq:colim_punctured_cube_associated_diagrams} are maps (i.e., $1$-cubes) of the form $\tilde{A}\rarrow\tilde{B}$ and $\capO_{\tilde{A}}[\mathbf{r}]\rarrow\capO_{\tilde{B}}[\mathbf{r}]$. Hence the case $n=0$ is verified by Proposition \ref{prop:OA_construction_on_n_connected_cofibrations}. Let $N\geq 1$ and assume the proposition is true for each $0\leq n\leq N-1$. Consider part (a); let's verify it remains true for $n=N$. Let $\function{i}{X}{Y}$ be a $k_{n+1}$-connected generating cofibration or generating acyclic cofibration in $\SymSeq$ with the positive flat stable model structure, $Z_0$ a pushout $n$-cube in $\LtO$, and consider any left-hand pushout diagram of the
form
\begin{align}
\label{eq:gluing_on_cells_higher_cube}
\xymatrix{
  \capO\circ X\ar[d]_-{\id\circ i}\ar[r] &
  {Z_0}_\emptyset\ar[d]\\
  \capO\circ Y\ar[r] &
  {Z_1}_\emptyset
}\quad\quad
\xymatrix{
  Z_0\ar[d] \\
  Z_1
}\quad\quad
\xymatrix{
  \capO_{Z_0}[\mathbf{r}]\ar[d]\\
  \capO_{Z_1}[\mathbf{r}]
}
\end{align}
in $\LtO$ with the middle map of pushout $n$-cubes the associated pushout $(n+1)$-cube in $\LtO$. Assume each ${Z_0}_\emptyset\rarrow {Z_0}_{\{i\}}$ is a $k_i$-connected cofibration between cofibrant objects in $\LtO$ $(1\leq i\leq n)$ and $\capO_{{Z_0}_\emptyset}$ is $(-1)$-connected; let's verify that the associated right-hand maps  of $n$-cubes $(r\geq 0)$, each regarded as an $(n+1)$-cube in $\SymSeq^{\Sigma_r^\op}$, are $(k_1+\dots+k_{n+1}+n)$-cocartesian. If $A:=Z_0$ and $\tilde{A}:=A_{\{1,\dots,n\}}$, then by Proposition \ref{prop:filtering_OA} there are corresponding filtrations
\begin{align}
\label{eq:induced_filtration_diagram_for_studying_induced_map_higher_cubes}
\xymatrix{
  \colim\limits_{\ \ \powerset_1(\mathbf{n})}\capO_A^0[\mathbf{r}]\ar[dd]^{\xi_0}\ar[r] &
  \colim\limits_{\ \ \powerset_1(\mathbf{n})}\capO_A^1[\mathbf{r}]\ar@{.>}[d]^{\xi_1}\ar[r] &
  \colim\limits_{\ \ \powerset_1(\mathbf{n})}\capO_A^2[\mathbf{r}]\ar@{.>}[d]^{\xi_2}\ar[r] &
  \dots\ar[r] &
  \colim\limits_{\ \ \powerset_1(\mathbf{n})}\capO_A^\infty[\mathbf{r}]\ar[d]^{\xi_\infty}\\
  &
  \cdot\ar@{.>}[d]^{(*)_1}\ar[r] &
  \cdot\ar@{.>}[d]^{(*)_2}\ar[r] &
  \dots\ar[r] &
  \cdot\ar[d]^{(*)_\infty}\\
  \capO_{\tilde{A}}^0[\mathbf{r}]\ar[r]\ar@/_0.5pc/[ur] &
  \capO_{\tilde{A}}^1[\mathbf{r}]\ar[r] &
  \capO_{\tilde{A}}^2[\mathbf{r}]\ar[r] &
  \dots\ar[r] &
  \capO_{\tilde{A}}^\infty[\mathbf{r}]
}
\end{align}
together with induced maps $\xi_t$ and $(*)_t$ ($t\geq 1$) that make the diagram in $\SymSeq^{\Sigma_r^\op}$ commute; here, the upper diagrams are pushout diagrams and $\xi_\infty:=\colim_t\xi_t$, the maps $(*)_t$ are the obvious induced maps and $(*)_\infty:=\colim_t(*)_t$, the left-hand vertical map is naturally isomorphic to
\begin{align*}
  \colim\limits_{\ \ \powerset_1(\mathbf{n})}\capO_{Z_0}[\mathbf{r}]\longrightarrow
  \capO_{\tilde{Z}_0}[\mathbf{r}],
\end{align*}
and the right-hand vertical maps are naturally isomorphic to the diagram
\begin{align}
\label{eq:right_hand_vertical_maps_higher_cubes}
  \colim\limits_{\ \ \powerset_1(\mathbf{n})}\capO_{Z_1}[\mathbf{r}]\longrightarrow
  \bigl(\colim\limits_{\ \ \powerset_1(\mathbf{n})}\capO_{Z_1}[\mathbf{r}]\bigr)\cup
  \capO_{\tilde{Z}_0}[\mathbf{r}]\longrightarrow
  \capO_{\tilde{Z}_1}[\mathbf{r}];
\end{align}
here, $\tilde{Z}_0:={Z_0}_{\{1,\dots,n\}}$ and $\tilde{Z}_1:={Z_1}_{\{1,\dots,n\}}$. We want to show that the right-hand map in \eqref{eq:right_hand_vertical_maps_higher_cubes} is $(k_1+\dots+k_{n+1}+n)$-connected; since the horizontal maps in \eqref{eq:induced_filtration_diagram_for_studying_induced_map_higher_cubes} are monomorphisms, it suffices to verify each map $(*)_t$ is $(k_1+\dots+k_{n+1}+n)$-connected. The argument is by induction on $t$. The map $\xi_0$ factors as
\begin{align*}
  \colim\limits_{\ \ \powerset_1(\mathbf{n})}\capO_A^0[\mathbf{r}]\longrightarrow
  \bigl(\colim\limits_{\ \ \powerset_1(\mathbf{n})}\capO_A^0[\mathbf{r}]\bigr)\cup\capO_{\tilde{A}}^0[\mathbf{r}]\xrightarrow[\Iso]{(*)_0}
  \capO_{\tilde{A}}^0[\mathbf{r}]
\end{align*}
and since the right-hand map $(*)_0$ is an isomorphism, it is $(k_1+\dots+k_{n+1}+n)$-connected. Consider the commutative diagram
\begin{align}
\label{eq:filtration_quotients_diagram_for_analyzing_connectivity_higher_cubes}
\xymatrix{
  \colim\limits_{\ \ \powerset_1(\mathbf{n})}
  \capO_A^{t-1}[\mathbf{r}]\ar[d]^{\xi_{t-1}}\ar[r] &
  \colim\limits_{\ \ \powerset_1(\mathbf{n})}
  \capO_A^t[\mathbf{r}]\ar[d]^{\xi_t}\ar[r] &
  \bigl(\colim\limits_{\ \ \powerset_1(\mathbf{n})}
  \capO_A[\mathbf{t+r}]\bigr)\tensorcheck_{\Sigma_t}(Y/X)^{\tensorcheck t}\ar[d]_{\Iso}\ar@/^2pc/[dd]^{(\#)}\\
  \cdot\ar[d]^{(*)_{t-1}}\ar[r] &
  \cdot\ar[d]^{(*)_t}\ar[r] &
  \cdot\ar[d]_{(\#\#)}\\
  \capO_{\tilde{A}}^{t-1}[\mathbf{r}]\ar[r] &
  \capO_{\tilde{A}}^t[\mathbf{r}]\ar[r] &
  \capO_{\tilde{A}}[\mathbf{t+r}]\tensorcheck_{\Sigma_t}(Y/X)^{\tensorcheck t}
}
\end{align}
with rows cofiber sequences. Since we know $(Y/X)^{\tensorcheck t}$ is at least $k_{n+1}$-connected and
$
  \colim_{\powerset_1(\mathbf{n})}\capO_A[\mathbf{t+r}]\longrightarrow
  \capO_{\tilde{A}}[\mathbf{t+r}]
$
is $(k_1+\dots+k_{n}+n-1)$-connected by the induction hypothesis, it follows that $(\#)$ is $(k_1+\dots+k_{n+1}+n)$-connected, and hence $(\#\#)$ is also. Since the rows in \eqref{eq:filtration_quotients_diagram_for_analyzing_connectivity_higher_cubes} are cofiber sequences, it follows by induction on $t$ that $(*)_t$ is $(k_1+\dots+k_{n+1}+n)$-connected for each $t\geq 1$. This finishes the argument that the the right-hand maps of $n$-cubes $(r\geq 0)$ in \eqref{eq:gluing_on_cells_higher_cube}, each regarded as an $(n+1)$-cube in $\SymSeq^{\Sigma_r^\op}$, are $(k_1+\dots+k_{n+1}+n)$-cocartesian.

Consider a sequence $Z_0\rarrow Z_1\rarrow Z_2\rarrow\cdots$ of pushout $n$-cubes in $\LtO$ as in \eqref{eq:gluing_on_cells_higher_cube}, define $\tilde{Z}_n:={Z_n}_{\{1,\dots,n\}}$, $Z_\infty:=\colim_nZ_n$, and $\tilde{Z}_\infty:=\colim_n\tilde{Z}_n$, and consider the naturally occurring map $Z_0\rarrow Z_\infty$ of pushout $n$-cubes, regarded as a pushout $(n+1)$-cube in $\LtO$. Consider the associated left-hand diagram of the form
\begin{align}
\label{eq:gluing_on_cells_filtration_sequence_O_construction_higher_cubes}
\xymatrix{
  \colim\limits_{\ \ \powerset_1(\mathbf{n})}Z_0\ar[d]\ar[r] & \tilde{Z}_0\ar[d]\\
  \colim\limits_{\ \ \powerset_1(\mathbf{n})}Z_\infty\ar[r] & \tilde{Z}_\infty
}\quad\quad
\xymatrix{
  \colim\limits_{\ \ \powerset_1(\mathbf{n})}\capO_{Z_0}[\mathbf{r}]\ar[d]\ar[r] &
  \capO_{\tilde{Z}_0}[\mathbf{r}]\ar[d]\\
  \colim\limits_{\ \ \powerset_1(\mathbf{n})}\capO_{Z_\infty}[\mathbf{r}]\ar[r] &
  \capO_{\tilde{Z}_\infty}[\mathbf{r}]
}
\end{align}
in the underlying category $\SymSeq$ and the associated right-hand diagrams ($r\geq 0$) in $\SymSeq^{\Sigma_r^\op}$. Assume each ${Z_0}_\emptyset\rarrow {Z_0}_{\{i\}}$ is a $k_i$-connected cofibration between cofibrant objects in $\LtO$ $(1\leq i\leq n)$ and $\capO_{{Z_0}_\emptyset}$ is $(-1)$-connected. We want to show that the right-hand diagrams in \eqref{eq:gluing_on_cells_filtration_sequence_O_construction_higher_cubes} are $(k_1+\dots + k_{n+1}+n)$-cocartesian. Consider the associated commutative diagram
\begin{align}
\label{eq:filtration_quotients_diagram_for_analyzing_connectivity_for_Z_higher_cubes}
\xymatrix{
  \colim\limits_{\ \ \powerset_1(\mathbf{n})}
  \capO_{Z_0}[\mathbf{r}]\ar[dd]\ar[r] &
  \colim\limits_{\ \ \powerset_1(\mathbf{n})}
  \capO_{Z_1}[\mathbf{r}]\ar@{.>}[d]^{\eta_1}\ar[r] &
  \colim\limits_{\ \ \powerset_1(\mathbf{n})}
  \capO_{Z_2}[\mathbf{r}]\ar@{.>}[d]^{\eta_2}\ar[r] &
  \cdots\ar[r] &
  \colim\limits_{\ \ \powerset_1(\mathbf{n})}
  \capO_{Z_\infty}[\mathbf{r}]\ar[d]^{\eta_\infty}\\
  &
  \cdot
  \ar@{.>}[d]^{(\#)_1}\ar[r] &
  \cdot
  \ar@{.>}[d]^{(\#)_2}\ar[r] &
  \cdots\ar[r] &
  \cdot\ar[d]^{(\#)_\infty}\\
  \capO_{\tilde{Z}_0}[\mathbf{r}]\ar[r]\ar@/_0.5pc/[ur] &
  \capO_{\tilde{Z}_1}[\mathbf{r}]\ar[r] &
  \capO_{\tilde{Z}_2}[\mathbf{r}]\ar[r] &
  \cdots\ar[r] &
  \capO_{\tilde{Z}_\infty}[\mathbf{r}]
}
\end{align}
in $\SymSeq^{\Sigma_r^\op}$ and induced maps $\eta_t$ and $(\#)_t$ $(t\geq 1)$; here, the upper diagrams are pushout diagrams and $\eta_\infty:=\colim_t\eta_t$, the maps $(\#)_t$ are the obvious induced maps and $(\#)_\infty:=\colim_t(\#)_t$, and the right-hand vertical maps are naturally isomorphic to the diagram
\begin{align}
\label{eq:right_hand_vertical_maps_sequential_diagram_higher_cubes}
  \colim\limits_{\ \ \powerset_1(\mathbf{n})}
  \capO_{Z_\infty}[\mathbf{r}]
  \longrightarrow
  \bigl(
  \colim\limits_{\ \ \powerset_1(\mathbf{n})}
  \capO_{Z_\infty}[\mathbf{r}]
  \bigr)
  \cup\capO_{\tilde{Z}_0}[\mathbf{r}]
  \longrightarrow
  \capO_{\tilde{Z}_\infty}[\mathbf{r}]
\end{align}
We want to show that the right-hand map in \eqref{eq:right_hand_vertical_maps_sequential_diagram_higher_cubes} is $(k_1+\dots + k_{n+1}+n)$-connected; since the horizontal maps in \eqref{eq:filtration_quotients_diagram_for_analyzing_connectivity_for_Z_higher_cubes} are monomorphisms, it suffices to verify each map $(\#)_t$ is $(k_1+\dots + k_{n+1}+n)$-connected. The argument is by induction on $t$. The map $(\#)_t$ factors as
\begin{align}
\label{eq:factorization_of_maps_to_analyze_higher_cubes}
\xymatrix{
  \bigl(
  \colim\limits_{\ \ \powerset_1(\mathbf{n})}
  \capO_{Z_t}[\mathbf{r}]
  \bigr)
  \cup\capO_{\tilde{Z}_0}[\mathbf{r}]\ar[r] &
  \bigl(
  \colim\limits_{\ \ \powerset_1(\mathbf{n})}
  \capO_{Z_t}[\mathbf{r}]
  \bigr)
  \cup\capO_{\tilde{Z}_{t-1}}[\mathbf{r}]\ar[r] &
  \capO_{\tilde{Z}_t}[\mathbf{r}]
}
\end{align}
We know from above that $(\#)_1$ and the right-hand map in \eqref{eq:factorization_of_maps_to_analyze_higher_cubes} are $(k_1+\dots + k_{n+1}+n)$-connected for each $t\geq 1$, hence it follows by induction on $t$ that $(\#)_t$ is $(k_1+\dots + k_{n+1}+n)$-connected for each $t\geq 1$. This finishes the argument that the right-hand diagrams $(r\geq 0)$ in \eqref{eq:gluing_on_cells_filtration_sequence_O_construction_higher_cubes} are $(k_1+\dots + k_{n+1}+n)$-cocartesian in  $\SymSeq^{\Sigma_r^\op}$.

It follows from Proposition \ref{prop:factorization_into_n_connected_cofibration} that the pushout $(n+1)$-cube $A\rarrow B$ factors as $Z_0\xrightarrow{i_\lambda} Z_\lambda \xrightarrow{p} B$, a composition of pushout $(n+1)$-cubes in $\LtO$, starting with $Z_0=A$, where $i_\lambda$ is a (possibly transfinite) composition of pushout $n$-cubes as in \eqref{eq:gluing_on_cells_higher_cube} and $p$ is an objectwise weak equivalence. Consider the associated diagram
\begin{align*}
\xymatrix{
  \colim\limits_{\ \ \powerset_1(\mathbf{n})}
  \capO_{Z_0}[\mathbf{r}]\ar[d]\ar[r] &
  \capO_{\tilde{Z}_0}[\mathbf{r}]\ar[d]\ar@/^0.5pc/[dr]\\
  \colim\limits_{\ \ \powerset_1(\mathbf{n})}
  \capO_{Z_\lambda}[\mathbf{r}]\ar[d]^{\wequiv}\ar[r] &
  \bigl(
  \colim\limits_{\ \ \powerset_1(\mathbf{n})}
  \capO_{Z_\lambda}[\mathbf{r}]
  \bigr)
  \cup\capO_{\tilde{Z}_0}[\mathbf{r}]\ar[d]^{\wequiv}\ar[r]^-{(**)} &
  \capO_{\tilde{Z}_\lambda}[\mathbf{r}]\ar[d]^{\wequiv}\\
  \colim\limits_{\ \ \powerset_1(\mathbf{n})}
  \capO_B[\mathbf{r}]\ar[r] &
  \bigl(
  \colim\limits_{\ \ \powerset_1(\mathbf{n})}
  \capO_B[\mathbf{r}]
  \bigr)
  \cup\capO_{\tilde{Z}_0}[\mathbf{r}]\ar[r]^-{(*)} &
  \capO_{\tilde{B}}[\mathbf{r}]
}
\end{align*}
Noting that the bottom vertical arrows are weak equivalences, it follows that $(*)$ has the same connectivity as $(**)$, which finishes the proof that the right-hand diagrams $(r\geq 0)$ in \eqref{eq:colim_punctured_cube_associated_diagrams} are $(k_1+\dots + k_{n+1}+n)$-cocartesian in $\SymSeq^{\Sigma_r^\op}$.  In particular, taking $r=0$ verifies that the left-hand diagram in \eqref{eq:colim_punctured_cube_associated_diagrams} is $(k_1+\dots + k_{n+1}+n)$-cocartesian in $\SymSeq$.
\end{proof}

\begin{thm}[Theorem \ref{thm:higher_homotopy_excision} restated]
Let $\capO$ be an operad in $\capR$-modules and $W$ a nonempty finite set. Let $\capX$ be a strongly $\infty$-cocartesian $W$-cube of $\capO$-algebras (resp. left $\capO$-modules). Assume that $\capR,\capO,\capX_\emptyset$ are $(-1)$-connected. Let $k_i\geq -1$ for each $i\in W$. If each $\capX_\emptyset\rarrow\capX_{\{i\}}$ is $k_i$-connected ($i\in W$),  then
\begin{itemize}
\item[(a)] $\capX$ is $l$-cocartesian in $\ModR$ (resp. $\SymSeq$) with $l=|W|-1+\sum_{i\in W}k_i$,
\item[(b)] $\capX$ is $k$-cartesian with $k=\sum_{i\in W}k_i$.
\end{itemize}
\end{thm}

\begin{proof}
It suffices to consider the case of left $\capO$-modules. It is enough to treat the special case where $\capX$ is a pushout cofibration $W$-cube in $\LtO$. The case $|W|=1$ is trivial and the case $|W|\geq 2$ follows from Theorem \ref{thm:pushout_cofibration_cube_homotopical_analysis}.
\end{proof}

\begin{thm}[Theorem \ref{thm:homotopy_excision} restated]
Let $\capO$ be an operad in $\capR$-modules. Let $\capX$ be a homotopy pushout square of $\capO$-algebras (resp. left $\capO$-modules) of the form
\begin{align*}
\xymatrix{
  \capX_\emptyset\ar[r]\ar[d] & \capX_{\{1\}}\ar[d]\\
  \capX_{\{2\}}\ar[r] & \capX_{\{1,2\}}
}
\end{align*}
Assume that $\capR,\capO,\capX_\emptyset$ are $(-1)$-connected. Consider any $k_1,k_2\geq -1$. If each $\capX_\emptyset\rarrow \capX_{\{i\}}$ is $k_i$-connected ($i=1,2$), then
\begin{itemize}
\item[(a)] $\capX$ is $l$-cocartesian in $\ModR$ (resp. $\SymSeq$) with $l=k_1+k_2 +1$,
\item[(b)] $\capX$ is $k$-cartesian with $k=k_1+k_2$.
\end{itemize}
\end{thm}

\begin{proof}
This is the special case $|W| = 2$ of Theorem \ref{thm:higher_homotopy_excision}.
\end{proof}

\subsection{Proof of the higher Blakers-Massey theorem for $\AlgO$ and $\LtO$}

The purpose of this section is to prove the Blakers-Massey theorems \ref{thm:blakers_massey} and \ref{thm:higher_blakers_massey}. We first show that Blakers-Massey for square diagrams (Theorem \ref{thm:blakers_massey}) follows fairly easily from the higher homotopy excision result proved in the previous section.

\begin{thm}[Theorem \ref{thm:blakers_massey} restated]
Let $\capO$ be an operad in $\capR$-modules. Let $\capX$ be a commutative square of $\capO$-algebras (resp. left $\capO$-modules) of the form
\begin{align*}
\xymatrix{
  \capX_\emptyset\ar[r]\ar[d] & \capX_{\{1\}}\ar[d]\\
  \capX_{\{2\}}\ar[r] & \capX_{\{1,2\}}
}
\end{align*}
Assume that $\capR,\capO,\capX_\emptyset$ are $(-1)$-connected. Consider any $k_1,k_2\geq -1$, and  $k_{12}\in\ZZ$. If each $\capX_\emptyset\rarrow \capX_{\{i\}}$ is $k_i$-connected ($i=1,2$) and $\capX$ is $k_{12}$-cocartesian, then $\capX$ is $k$-cartesian, where $k$ is the minimum of $k_{12}-1$ and $k_{1}+k_{2}$.
\end{thm}

\begin{proof}
It suffices to consider the case of left $\capO$-modules. Let $W:=\{1,2\}$. It is enough to consider the special case where $\capX$ is a cofibration $W$-cube in $\LtO$. Consider the induced maps
\begin{align*}
\xymatrix{
  \colim\nolimits^{\SymSeq}_{\powerset_1(W)}\capX\ar[r]^-{(*)} &
  \colim\nolimits^{\LtO}_{\powerset_1(W)}\capX\ar[r]^-{(**)} &
  \colim\nolimits^{\LtO}_{\powerset(W)}\Iso\capX_W
}
\end{align*}
We know that $(*)$ is $(k_1+k_2+1)$-connected by homotopy excision (Theorem \ref{thm:homotopy_excision}) and $(**)$ is $k_{12}$-connected by assumption. Hence by Proposition \ref{prop:connectivity_estimates_for_composition_of_maps}(a) the composition is $l$-connected, where $l$ is the minimum of $k_1+k_2+1$ and $k_{12}$; in other words, we have verified that $\capX$ is $l$-cocartesian in $\SymSeq$, and Proposition \ref{prop:comparing_cocartesian_and_cartesian_estimates_in_ModR}(a) finishes the proof.

\end{proof}

We now turn to the proof of the higher Blakers-Massey result (Theorem~\ref{thm:higher_blakers_massey}). Our approach follows that used by Goodwillie at the corresponding point in \cite{Goodwillie_calc2}.

The following is an important warm-up calculation for Proposition \ref{prop:cocartesian_estimates_for_induction_argument}.

\begin{prop}
\label{prop:warmup_cocartesian_estimates_for_induction_argument}
Let $\capO$ be an operad in $\capR$-modules and $W$ a nonempty finite set. Let $\capX$ be a cofibration $W$-cube of $\capO$-algebras (resp. left $\capO$-modules). Assume that
\begin{itemize}
\item[(i)] for each nonempty subset $V\subset W$, the $V$-cube $\partial_\emptyset^V\capX$ (formed by all maps in $\capX$ between $\capX_\emptyset$ and $\capX_V$) is $k_V$-cocartesian,
\item[(ii)] $k_{U}\leq k_V$ for each $U\subset V$.
\end{itemize}
Then, for every $U\subsetneqq V\subset W$, the $(V-U)$-cube $\partial_U^V\capX$ is $k_{V-U}$-cocartesian.
\end{prop}

\begin{proof}
The argument is by induction on $|U|$. The case $|U|=0$ is true by assumption. Let $n\geq 1$ and assume the proposition is true for each $0\leq|U|<n$. Let's verify it remains true for $|U|=n$. Let $u\in U$ and note that $\partial_{U-\{u\}}^V\capX$ can be written as the composition of cubes
\begin{align*}
\xymatrix{
  \partial_{U-\{u\}}^{V-\{u\}}\capX\ar[r] &
  \partial_U^V\capX
}
\end{align*}
We know by the induction assumption that the composition of cubes is $k_{(V-U)\cup\{u\}}$-cocartesian and the left-hand cube is $k_{V-U}$-cocartesian. Since $k_{V-U}\leq k_{(V-U)\cup\{u\}}$ by assumption, it follows from Proposition \ref{prop:map_of_cubical_diagrams}(a) that the right-hand cube is $k_{V-U}$-cocartesian which finishes the proof.
\end{proof}

\begin{defn}
Let $\capO$ be an operad in $\capR$-modules and $W$ a nonempty finite set. Let $\capX$ be a $W$-cube of $\capO$-algebras (resp. left $\capO$-modules) and consider any subset $\capB\subset\powerset(W)$.
\begin{itemize}
\item A subset $\capA\subset\capB$ is \emph{convex} if every element of $\capB$ which is less than an element of $\capA$ is in $\capA$.
\item Define $\capA^W_\mathrm{min}:=\{V\subset W: |V|\leq 1\}$ and $\capA^W_\mathrm{max}:=\powerset(W)$.
\item For each convex subset $\capA\subset\powerset(W)$, the $W$-cube $\capX_\capA$ is defined objectwise by $(\capX_\capA)_U := \colim_{\capA\cap\powerset(U)}\capX\Iso\colim_{T\in\capA,\, T\subset U}\capX_T$.
\end{itemize}
\end{defn}

The following proposition will be needed in the proof of Proposition \ref{prop:properties_of_the_X_sub_A_construction} below.

\begin{prop}
\label{prop:colim_over_punctured_cube_of_X_sub_A_construction}
Let $\capO$ be an operad in $\capR$-modules and $W$ a finite set. Let $\capX$ be a $W$-cube of $\capO$-algebras (resp. left $\capO$-modules) and consider any convex subset $\capA\subset\powerset(W)$. Then for each nonempty subset $V\subset W$, there is a natural isomorphism
\begin{align*}
  \colim_{\powerset_1(V)}\capX_\capA\Iso
  \colim_{\capA\cap\powerset_1(V)}\capX=
  \colim_{T\in\capA,\,T\subsetneqq V}\capX_T
\end{align*}
\end{prop}

\begin{proof}
This is because the indexing sets $\{T\in\capA:T\subset U,\, U\subsetneqq V\}$ and $\{T\in\capA:T\subsetneqq V\}$ are the same.
\end{proof}

The following proposition explains the key properties of the $\capX_\capA$ construction and its relationship to $\capX$; it is through these properties that the $\capX_\capA$ construction is useful and meaningful.

\begin{prop}
\label{prop:properties_of_the_X_sub_A_construction}
Let $\capO$ be an operad in $\capR$-modules and $W$ a nonempty finite set. Let $\capX$ be a $W$-cube of $\capO$-algebras (resp. left $\capO$-modules) and consider any convex subset $\capA^W_\mathrm{min}\subset\capA\subset\capA^W_\mathrm{max}$.
\begin{itemize}
\item[(a)] There are natural isomorphisms $\partial_\emptyset^V\capX_\capA\Iso\partial_\emptyset^V\capX$ if $V\in\capA$,
\item[(b)] The $V$-cube $\partial_\emptyset^V\capX_\capA$ is $\infty$-cocartesian if $V\notin\capA$ and $\capX$ is a cofibration $W$-cube in $\AlgO$.
\end{itemize}
\end{prop}

\begin{proof}
Consider part (a). Let $V\in\capA$. Then we know that $\powerset(V)\subset\capA$, by $\capA$ convex, and hence $\capA\cap\powerset(V)=\powerset(V)$. It follows that $(\capX_\capA)_V=\colim_{\powerset(V)}\capX\Iso\capX_V$. The same argument shows that $(\capX_\capA)_{V'}\Iso\capX_{V'}$, for each $V'\subset V$. Consider part (b). Let $V\subset W$ and $V\notin\capA$. Then $\capA\cap\powerset(V)=\capA\cap\powerset_1(V)$ and hence the composition
\begin{align*}
  \colim_{\powerset_1(V)}\partial_\emptyset^V\capX_\capA=
  \colim_{\powerset_1(V)}\capX_\capA\Iso
  \colim_{\capA\cap\powerset_1(V)}\capX=
  \colim_{\capA\cap\powerset(V)}\capX=
  (\capX_\capA)_V=(\partial_\emptyset^V\capX_\capA)_V
\end{align*}
is an isomorphism which finishes the proof of part (b).
\end{proof}

The following proposition shows that $\capX_\capA$ inherits several of the homotopical properties of $\capX$.

\begin{prop}
\label{prop:cocartesian_estimates_for_induction_argument}
Let $\capO$ be an operad in $\capR$-modules and $W$ a nonempty finite set. Let $\capX$ be a cofibration $W$-cube of $\capO$-algebras (resp. left $\capO$-modules) and consider any convex subset $\capA^W_\mathrm{min}\subset\capA\subset\capA^W_\mathrm{max}$. Assume that
\begin{itemize}
\item[(i)] for each nonempty subset $V\subset W$, the $V$-cube $\partial_\emptyset^V\capX$ (formed by all maps in $\capX$ between $\capX_\emptyset$ and $\capX_V$) is $k_V$-cocartesian,
\item[(ii)] $k_{U}\leq k_V$ for each $U\subset V$.
\end{itemize}
Then, for each $U\subsetneqq V\subset W$, the $(V-U)$-cube $\partial_U^V\capX_\capA$ is $k_{V-U}$-cocartesian.
\end{prop}

\begin{proof}
By the induction argument in the proof of Proposition \ref{prop:warmup_cocartesian_estimates_for_induction_argument}, it suffices to verify that the $V$-cube $\partial_\emptyset^V\capX_\capA$ is $k_{V}$-cocartesian, for each nonempty subset $V\subset W$, and this follows immediately from Proposition \ref{prop:properties_of_the_X_sub_A_construction}.
\end{proof}

The following proposition is proved in Goodwillie \cite[2.8]{Goodwillie_calc2} in the context of spaces, and exactly the same argument gives a proof in the context of structured ring spectra; this is an exercise left to the reader.

\begin{prop}
\label{prop:cofibration_cube_properties}
Let $\capO$ be an operad in $\capR$-modules and $W$ a nonempty finite set. Let $\capX$ be a cofibration $W$-cube of $\capO$-algebras (resp. left $\capO$-modules) and consider any convex subset $\capA\subset\powerset(W)$.
\begin{itemize}
\item[(a)] For each inclusion $\capA'\subset\capA$ of convex subsets of $\powerset(W)$, the induced map $\colim_{\capA'}\capX\rarrow\colim_{\capA}\capX$ is a cofibration in $\AlgO$ (resp. $\LtO$),
\item[(b)] For any convex subsets $\capA,\capB$ of $\powerset(W)$, the diagram
\begin{align*}
\xymatrix{
  \colim_{\capA\cap\capB}\capX\ar[r]\ar[d] & \colim_{\capB}\capX\ar[d]\\
  \colim_{\capA}\capX\ar[r] & \colim_{\capA\cup\capB}\capX
}
\end{align*}
is a pushout diagram of cofibrations in $\AlgO$ (resp. $\LtO$).
\item[(c)] If $A\in\capA$ is maximal and the cofibration
$
  \colim_{\powerset_1(A)}\capX\rarrow\colim_{\powerset(A)}\capX\Iso\capX_A
$
is $k_A$-connected, then the cofibration
$$
  \colim_{\capA-\{A\}}\capX\rarrow\colim_{\capA}\capX
$$
is $k_A$-connected.
\end{itemize}
\end{prop}

The purpose of the following induction argument is to leverage the higher homotopy excision result (Theorem \ref{thm:higher_homotopy_excision}) for structured ring spectra into a proof of the first main theorem in this paper (Theorem \ref{thm:higher_blakers_massey})---the higher Blakers-Massey theorem for structured ring spectra. Proposition \ref{prop:cubical_induction_argument} is motivated by Goodwillie \cite[2.12]{Goodwillie_calc2}; it is essentially Goodwillie's cubical induction argument, appropriately modified to our situation.

\begin{prop}[Cubical induction argument]
\label{prop:cubical_induction_argument}
Let $\capO$ be an operad in $\capR$-modules and $W$ a nonempty finite set. Suppose $\capA\subset\capA_\mathrm{max}^W$ is convex, $A\in\capA$ maximal, $|A|\geq 2$, and $\capA':=\capA-\{A\}$. Let $\capX$ be a cofibration $W$-cube of $\capO$-algebras (resp. left $\capO$-modules). Assume that $\capR,\capO,\capX_\emptyset$ are $(-1)$-connected, and suppose that
\begin{itemize}
\item[(i)] for each nonempty subset $V\subset W$, the $V$-cube $\partial_\emptyset^V\capX$ (formed by all maps in $\capX$ between $\capX_\emptyset$ and $\capX_V$) is $k_V$-cocartesian,
\item[(ii)] $-1\leq k_{U}\leq k_V$ for each $U\subset V$.
\end{itemize}
If $X_{\capA'}$ is $k$-cartesian, then $\capX_\capA$ is $k$-cartesian, where $k$ is the minimum of $-|W|+\sum_{V\in\lambda}(k_V+1)$ over all partitions $\lambda$ of $W$ by nonempty sets.
\end{prop}

\begin{rem}
\label{rem:3_cube_example_of_cubical_induction}
In the case that $\capX$ is a $3$-cube, the following diagram illustrates one of the cubical decompositions covered by Proposition \ref{prop:cubical_induction_argument}. It corresponds to the sequence of maximal elements: $\{1,2\}$, $\{1,3\}$, $\{2,3\}$, $\{1,2,3\}$; we benefitted from the discussion in Munson-Voli\'c \cite{Munson_Volic_book_project} where this cubical decomposition is used, in the context of spaces, to illustrate Goodwillie's cubical induction argument.

\begin{align*}
\xymatrix@!0{
\capX_\emptyset\ar[rr]\ar[dd]\ar[dr] &&
\capX_{\{1\}}\ar[dr]\ar'[d][dd]\ar@{=}[rr] &&
\capX_{\{1\}}\ar@{=}[rr]\ar[dr]\ar'[d][dd] &&
\capX_{\{1\}}\ar[dr]\ar[dd]|(0.5)\hole|(0.53)\hole\\
&\capX_{\{2\}}\ar[rr]\ar[dd] &&
\cdot\ar[dd]\ar[rr] &&
\capX_{\{1,2\}}\ar@{=}[rr]\ar[dd] &&
\capX_{\{1,2\}}\ar[dd]\\
\capX_{\{3\}}\ar[dr]\ar'[r][rr]\ar@{=}[dd] &&
\cdot\ar[dr]\ar@{=}'[r][rr] &&
\cdot\ar[dr]\ar'[r][rr] &&
\capX_{\{1,3\}}\ar[dr]\ar@{=}'[d][dd]\\
&\cdot\ar[rr]\ar[dd] &&
\cdot\ar[rr] &&
\cdot\ar[rr] &&
\cdot\ar[dd]\\
\capX_{\{3\}}\ar[dr]\ar'[r][rrrrrr]\ar@{=}[dd]&&&&&&
\capX_{\{1,3\}}\ar[dr]\ar@{=}'[d][dd]\\
&\capX_{\{2,3\}}\ar[rrrrrr]\ar@{=}[dd]&&&&&&
\cdot\ar[dd]\\
\capX_{\{3\}}\ar[dr]\ar'[r][rrrrrr] &&&&&&
\capX_{\{1,3\}}\ar[dr]\\
&\capX_{\{2,3\}}\ar[rrrrrr]&&&&&&\capX_{\{1,2,3\}}
}
\end{align*}
\end{rem}

\begin{proof}[Proof of Proposition \ref{prop:cubical_induction_argument}]
The argument is by induction on $|W|$. The case $|W|=2$ is verified by the proof of Theorem \ref{thm:blakers_massey}. Let $n\geq 3$ and assume the proposition is true for each $2\leq |W|<n$. Let's verify it remains true for $|W|=n$.

We know by assumption that $\capX_{\capA'}$ is $k$-cartesian. We want to verify that $\capX_\capA$ is $k$-cartesian (it might be helpful at this point to look ahead to  \eqref{eq:final_decomposition_for_induction_argument} for the decomposition of $\capX_\capA$ that we will use to finish the proof). Consider the induced map of $W$-cubes $\capX_{\capA'}\rarrow\capX_\capA$. Note that if $V\not\supset A$, then $\powerset(V)\not\ni A$ and hence $\capA'\cap\powerset(V)=\capA\cap\powerset(V)$; in particular, each of the maps
\begin{align}
\label{eq:subdiagram_of_identity_maps}
\xymatrix{
  (\capX_{\capA'})_V\ar[r]^-{\id} & (\capX_\capA)_V,
  & V\not\supset A
}
\end{align}
in $\capX_{\capA'}\rarrow\capX_\capA$ is the identity map. Note that if $V\supset U\supset A$, then the diagram
\begin{align}
\label{eq:pushout_diagram_in_induction_step}
\xymatrix{
  (\capX_{\capA'})_U\ar[r]\ar[d] & (\capX_\capA)_U\ar[d]\\
  (\capX_{\capA'})_V\ar[r] & (\capX_\capA)_V,&
  V\supset U\supset A
}
\end{align}
is a pushout diagram by Proposition \ref{prop:cofibration_cube_properties}(b); in particular, focusing on this case is the same as focusing on the subdiagram $\partial_A^W\capX_{\capA'}\rarrow\partial_A^W\capX_\capA$ of $\capX_{\capA'}\rarrow\capX_\capA$.

Let's first verify that $\partial_A^W\capX_{\capA'}\rarrow\partial_A^W\capX_\capA$ is $(k+|A|-1)$-cartesian. Let $\capY$ denote $\partial_A^W\capX_{\capA'}\rarrow\partial_A^W\capX_\capA$ regarded as a $((W-A)\cup\{*\})$-cube as follows:
\begin{align*}
  \capY_V&:=(\partial_A^W\capX_{\capA'})_V=(\capX_{\capA'})_{V\cup A},\quad\ V\subset W-A,\\
  \capY_V&:=(\partial_A^W\capX_\capA)_U=(\capX_\capA)_{U\cup A},\quad\quad U\subset W-A,\quad V=U\cup\{*\}.
\end{align*}
Since $|(W-A)\cup\{*\}|<|W|$, our induction assumption can be applied to $\capY$, provided that the appropriate $k'_V$-cocartesian estimates are satisfied. We claim that with the following definitions
\begin{align*}
  k'_V&:=k_V,\quad\quad \emptyset\neq V\subset W-A,\\
  k'_V&:=k_A,\quad\quad V=\{*\}\\
  k'_V&:=\infty,\quad\quad \emptyset\neq U\subset W-A,\quad V=U\cup\{*\},
\end{align*}
the cube $\partial_\emptyset^V\capY$ (formed by all maps in $\capY$ between $\capY_\emptyset$ and $\capY_V$) is $k'_V$-cocartesian, for each nonempty subset $V\subset (W-A)\cup\{*\}$. Let's verify this now: If $V\subset W-A$ is nonempty, then $\partial_\emptyset^V\capY$ is the cube $\partial_A^{A\cup V}\capX_{\capA'}$ which is $k_V$-cocartesian by Proposition \ref{prop:cocartesian_estimates_for_induction_argument}. If $V=\{*\}$, then $\partial_\emptyset^V\capY$ is the map $(\capX_{\capA'})_A\rarrow(\capX_\capA)_A$ which is $k_A$-connected by Proposition \ref{prop:cofibration_cube_properties}(c). Finally, if $V=U\cup\{*\}$ for any nonempty $U\subset W-A$, then $\partial_\emptyset^V\capY$ is the cube $\partial_A^{A\cup U}\capX_{\capA'}\rarrow\partial_A^{A\cup U}\capX_\capA$ which is $\infty$-cocartesian by
\eqref{eq:pushout_diagram_in_induction_step} and Proposition \ref{prop:map_of_cubical_diagrams}(b). Hence by our induction hypothesis applied to $\capY$: since the sum $\sum_{V\in\lambda'}(k'_V+1)$ for a partition $\lambda'$ of $(W-A)\cup\{*\}$ by nonempty sets is always either $\infty$, or the sum $\sum_{U\in\lambda}(k_U+1)$ for a partition $\lambda$ of $W$ by nonempty sets in which some $U$ is $A$, we know that $\capY$ is $k'$-cartesian, where $k'$ is the minimum of $-|(W-A)\cup\{*\}|+\sum_{U\in\lambda}(k_U+1)$ over all partitions $\lambda$ of $W$ by nonempty sets in which some $U$ is $A$. In particular, this implies that $\capY$, and hence $\partial_A^W\capX_{\capA'}\rarrow\partial_A^W\capX_\capA$, is $(k+|A|-1)$-cartesian.

We next want to verify that
\begin{align}
\label{eq:induction_argument_cartesian_estimate_for_certain_subdiagrams}
  \partial_{A'}^W\capX_{\capA'}\rarrow\partial_A^W\capX_\capA
  \quad \text{is $(k+|A'|-1)$-cartesian}
\end{align}
for each $A'\subset A$. We know that \eqref{eq:induction_argument_cartesian_estimate_for_certain_subdiagrams} is true for $A'=A$ by above. We will argue by downward induction on $|A'|$. Suppose that \eqref{eq:induction_argument_cartesian_estimate_for_certain_subdiagrams} is true for some nonempty $A'\subset A$. Let $a\in A'$ and note that the cube $\partial_{A'-\{a\}}^W\capX_{\capA'}\rarrow\partial_{A'-\{a\}}^W\capX_\capA$ can be written as the diagram of cubes
\begin{align}
\label{eq:downward_induction_cube_factorization}
\xymatrix{
  \partial_{A'-\{a\}}^{W-\{a\}}\capX_{\capA'}\ar[r]^{\id}\ar[d] &
  \partial_{A'-\{a\}}^{W-\{a\}}\capX_\capA\ar[d]\\
  \partial_{A'}^W\capX_{\capA'}\ar[r] &
  \partial_{A'}^W\capX_{\capA}
}
\end{align}
We know the top arrow is the identity by \eqref{eq:subdiagram_of_identity_maps}, and the bottom arrow is $(k+|A'|-1)$-cartesian by assumption. It follows from  Proposition \ref{prop:map_of_cubical_diagrams}(d) that  \eqref{eq:downward_induction_cube_factorization} is $(k+|A'|-2)$-cartesian, which finishes the argument that \eqref{eq:induction_argument_cartesian_estimate_for_certain_subdiagrams} is true for each $A'\subset A$.

To finish off the proof, let $a\in A$ and note by \eqref{eq:subdiagram_of_identity_maps} that $\capX_\capA$ can be written as the composition of cubes
\begin{align}
\label{eq:final_decomposition_for_induction_argument}
  \partial_\emptyset^{W-\{a\}}\capX_{\capA'}\rarrow
  \partial_{\{a\}}^W\capX_{\capA'}\rarrow
  \partial_{\{a\}}^W\capX_\capA.
\end{align}
The right-hand arrow is $k$-cartesian by \eqref{eq:induction_argument_cartesian_estimate_for_certain_subdiagrams} and the left-hand arrow is $\capX_{\capA'}$ which is $k$-cartesian by assumption, hence by Proposition \ref{prop:composed_map_of_cubical_diagrams}(c) it follows that $\capX_\capA$ is $k$-cartesian which finishes the proof.
\end{proof}

\begin{thm}[Theorem \ref{thm:higher_blakers_massey} restated]
Let $\capO$ be an operad in $\capR$-modules and $W$ a nonempty finite set. Let $\capX$ be a $W$-cube of $\capO$-algebras (resp. left $\capO$-modules). Assume that $\capR,\capO,\capX_\emptyset$ are $(-1)$-connected, and suppose that
\begin{itemize}
\item[(i)] for each nonempty subset $V\subset W$, the $V$-cube $\partial_\emptyset^V\capX$ (formed by all maps in $\capX$ between $\capX_\emptyset$ and $\capX_V$) is $k_V$-cocartesian,
\item[(ii)] $-1\leq k_{U}\leq k_V$ for each $U\subset V$.
\end{itemize}
Then $\capX$ is $k$-cartesian, where $k$ is the minimum of $-|W|+\sum_{V\in\lambda}(k_V+1)$ over all partitions $\lambda$ of $W$ by nonempty sets.
\end{thm}

\begin{proof}
It suffices to consider the case of left $\capO$-modules. It is enough to consider the special case where $\capX$ is a cofibration $W$-cube in $\LtO$. Let $\capA:=\capA_{\mathrm{max}}^W$, $\capA':=\capA_{\mathrm{max}}^W-\{W\}$, and note that $\capX_\capA$ is equal to $\capX$. Then it follows by induction from Proposition \ref{prop:cubical_induction_argument}, together with Theorem \ref{thm:higher_homotopy_excision} (to start the induction using $\capA'=\capA^W_{\mathrm{min}}$) that $\capX$ is $k$-cartesian.
\end{proof}

\subsection{Proof of higher dual homotopy excision for $\AlgO$ and $\LtO$}

We now turn to the dual versions of our main results. In this section we prove the dual homotopy excision results (Theorems \ref{thm:dual_homotopy_excision} and \ref{thm:higher_dual_homotopy_excision}). Notice that here we are leveraging the fact that cartesian-ness in the categories $\AlgO$ and $\LtO$ is detected in the underlying categories of $\capR$-modules and symmetric sequences, and that, in those underlying categories, there is a close relationship between cartesian-ness and cocartesian-ness, given by Proposition~\ref{prop:comparing_cocartesian_and_cartesian_estimates_in_ModR}.

\begin{thm}[Theorem \ref{thm:dual_homotopy_excision} restated]
Let $\capO$ be an operad in $\capR$-modules. Let $\capX$ be a homotopy pullback square of $\capO$-algebras (resp. left $\capO$-modules) of the form
\begin{align*}
\xymatrix{
  \capX_\emptyset\ar[r]\ar[d] & \capX_{\{1\}}\ar[d]\\
  \capX_{\{2\}}\ar[r] & \capX_{\{1,2\}}
}
\end{align*}
Assume that $\capR,\capO,\capX_\emptyset$ are $(-1)$-connected. Consider any $k_1,k_2\geq -1$. If $\capX_{\{2\}}\rarrow \capX_{\{1,2\}}$ is $k_1$-connected and $\capX_{\{1\}}\rarrow \capX_{\{1,2\}}$ is $k_2$-connected, then $\capX$ is $k$-cocartesian with $k=k_{1}+k_{2}+2$.
\end{thm}

\begin{proof}
It suffices to consider the case of left $\capO$-modules. Let $W:=\{1,2\}$. It is enough to consider the special case where $\capX$ is a cofibration $W$-cube in $\LtO$. Consider the induced maps
\begin{align*}
\xymatrix{
  \colim\nolimits^{\SymSeq}_{\powerset_1(W)}\capX\ar[r]^-{(*)} &
  \colim\nolimits^{\LtO}_{\powerset_1(W)}\capX\ar[r]^-{(**)} &
  \colim\nolimits^{\LtO}_{\powerset(W)}\Iso\capX_W
}
\end{align*}
We know that $(*)$ is $(k_1+k_2+1)$-connected by homotopy excision (Theorem \ref{thm:homotopy_excision}), and since $\capX$ is $\infty$-cocartesian in the underlying category $\SymSeq$, the composition is $\infty$-connected. Hence by Proposition \ref{prop:connectivity_estimates_for_composition_of_maps}(b) the map $(**)$ is $(k_1+k_2+2)$-connected which finishes the proof.
\end{proof}

\begin{thm}[Theorem \ref{thm:higher_dual_homotopy_excision} restated]
Let $\capO$ be an operad in $\capR$-modules and $W$ a finite set with $|W|\geq 2$. Let $\capX$ be a strongly $\infty$-cartesian $W$-cube of $\capO$-algebras (resp. left $\capO$-modules). Assume that $\capR,\capO,\capX_\emptyset$ are $(-1)$-connected. Let $k_i\geq -1$ for each $i\in W$. If each $\capX_{W-\{i\}}\rarrow\capX_W$ is $k_i$-connected ($i\in W$),  then $\capX$ is $k$-cocartesian with $k=|W|+\sum_{i\in W}k_i$.
\end{thm}

\begin{proof}
It suffices to consider the case of left $\capO$-modules. The argument is by induction on $|W|$. The case $|W|=2$ is verified by Theorem \ref{thm:dual_homotopy_excision}. Let $n\geq 3$ and assume the theorem is true for each $2\leq |W|< n$. Let's verify it remains true for $|W|=n$.

It suffices to consider the special case where $\capX$ is a cofibration $W$-cube in $\LtO$. Consider the induced maps
\begin{align}
\label{eq:composition_of_maps_for_determining_cocartesian_estimate}
\xymatrix{
  \colim\nolimits^{\SymSeq}_{\powerset_1(W)}\capX\ar[r]^-{(*)} &
  \colim\nolimits^{\LtO}_{\powerset_1(W)}\capX\ar[r]^-{(**)} &
  \colim\nolimits^{\LtO}_{\powerset(W)}\Iso\capX_W
}
\end{align}
We want to show that $(**)$ is $k$-connected. Consider the $W$-cube $\capX':=\capX_{\capA_\mathrm{max}-\{W\}}$ in $\LtO$ and note that $(*)$ is isomorphic to the map
\begin{align*}
\xymatrix{
  \colim\nolimits^{\SymSeq}_{\powerset_1(W)}\capX'\ar[r]^-{(*)'} &
  \colim\nolimits^{\SymSeq}_{\powerset(W)}\capX'\Iso\capX'_W
}
\end{align*}
We know that $\capX'$ is $\infty$-cocartesian in $\LtO$, hence by higher Blakers-Massey (Theorem \ref{thm:higher_blakers_massey}) applied to $\capX'$, together with the induction hypothesis, it follows that $\capX'$ is $k'$-cartesian with $k'=\sum_{i\in W}k_i$. Hence by Proposition \ref{prop:comparing_cocartesian_and_cartesian_estimates_in_ModR} we know that $\capX'$ is $(k'+|W|-1)$-cocartesian in $\SymSeq$, and therefore $(*)$ is $(k'+|W|-1)$-connected. Since the composition in \eqref{eq:composition_of_maps_for_determining_cocartesian_estimate} is $\infty$-connected, it follows from Proposition \ref{prop:connectivity_estimates_for_composition_of_maps}(b) that $(**)$ is $(k'+|W|)$-connected which finishes the proof.
\end{proof}

\subsection{Proof of the higher dual Blakers-Massey theorem for $\AlgO$ and $\LtO$}

The purpose of this section is to prove Theorems \ref{thm:dual_blakers_massey} and \ref{thm:higher_dual_blakers_massey}.

\begin{thm}[Theorem \ref{thm:dual_blakers_massey} restated]
Let $\capO$ be an operad in $\capR$-modules. Let $\capX$ be a commutative square of $\capO$-algebras (resp. left $\capO$-modules) of the form
\begin{align*}
\xymatrix{
  \capX_\emptyset\ar[r]\ar[d] & \capX_{\{1\}}\ar[d]\\
  \capX_{\{2\}}\ar[r] & \capX_{\{1,2\}}
}
\end{align*}
Assume that $\capR,\capO,\capX_\emptyset$ are $(-1)$-connected. Consider any $k_1,k_2,k_{12}\geq -1$ with $k_1\leq k_{12}$ and $k_2\leq k_{12}$. If $\capX_{\{2\}}\rarrow \capX_{\{1,2\}}$ is $k_1$-connected, $\capX_{\{1\}}\rarrow \capX_{\{1,2\}}$ is $k_2$-connected, and $\capX$ is $k_{12}$-cartesian, then $\capX$ is $k$-cocartesian, where $k$ is the minimum of $k_{12}+1$ and $k_{1}+k_{2}+2$.
\end{thm}

\begin{proof}
It suffices to consider the case of left $\capO$-modules. Let $W:=\{1,2\}$. It is enough to consider the special case where $\capX$ is a cofibration $W$-cube in $\LtO$. Consider the induced maps
\begin{align}
\label{eq:composition_of_maps_in_dual_blakers_massey}
\xymatrix{
  \colim\nolimits^{\SymSeq}_{\powerset_1(W)}\capX\ar[r]^-{(*)} &
  \colim\nolimits^{\LtO}_{\powerset_1(W)}\capX\ar[r]^-{(**)} &
  \colim\nolimits^{\LtO}_{\powerset(W)}\Iso\capX_W
}
\end{align}
Since $\capX$ is $k_{12}$-cartesian, we know by Proposition \ref{prop:comparing_cocartesian_and_cartesian_estimates_in_ModR} that $\capX$ is $(k_{12}+1)$-cocartesian in $\SymSeq$, and hence the composition in \eqref{eq:composition_of_maps_in_dual_blakers_massey} is $(k_{12}+1)$-connected. Since we know the map $(*)$ is $(k_1+k_2+1)$-connected by homotopy excision (Theorem \ref{thm:homotopy_excision}), it follows that $(**)$ is $k$-connected by Proposition \ref{prop:connectivity_estimates_for_composition_of_maps}(b) which finishes the proof.
\end{proof}

\begin{defn}
Let $\capO$ be an operad in $\capR$-modules and $W$ a nonempty finite set. Let $\capX$ be a $W$-cube of $\capO$-algebras (resp. left $\capO$-modules) and consider any subset $\capB\subset\powerset(W)$.
\begin{itemize}
\item A subset $\capA\subset\capB$ is \emph{concave} if every element of $\capB$ which is greater than an element of $\capA$ is in $\capA$.
\item Define $\capA_W^\mathrm{min}:=\{V\subset W: |W-V|\leq 1\}$ and $\capA_W^\mathrm{max}:=\powerset(W)$.
\item For each concave subset $\capA\subset\powerset(W)$, the $W$-cube $\capX^\capA$ is defined objectwise by $(\capX^\capA)_U := \lim_{T\in\capA,\, T\supset U}\capX_T$.
\end{itemize}
\end{defn}

\begin{prop}
\label{prop:warmup_cartesian_estimates_for_induction_argument}
Let $\capO$ be an operad in $\capR$-modules and $W$ a nonempty finite set. Let $\capX$ be a fibration $W$-cube of $\capO$-algebras (resp. left $\capO$-modules). Assume that
\begin{itemize}
\item[(i)] for each nonempty subset $V\subset W$, the $V$-cube $\partial_{W-V}^W\capX$  (formed by all maps in $\capX$ between $\capX_{W-V}$ and $\capX_W$) is $k_V$-cartesian,
\item[(ii)] $k_{U}\leq k_V$ for each $U\subset V$.
\end{itemize}
Then, for every $U\subsetneqq V\subset W$, the $(V-U)$-cube $\partial_U^V\capX$ is $k_{V-U}$-cartesian.
\end{prop}

\begin{proof}
It will be useful to note that assumptions (i) and (ii) are equivalent to the following assumptions:
\begin{itemize}
\item[(1)] for each subset $U\subsetneqq W$, the $(W-U)$-cube $\partial_{U}^W\capX$  (formed by all maps in $\capX$ between $\capX_U$ and $\capX_W$) is $k_{W-U}$-cartesian,
\item[(2)] $k_{W-V}\leq k_{W-U}$ for each $U\subset V$.
\end{itemize}
We want to verify that $\partial_U^V\capX$ is $k_{V-U}$-cartesian for each $U\subsetneqq V\subset W$. The argument is by downward induction on $|V|$. The case $|V|=|W|$ is true by assumption. Assume the proposition is true for some nonempty $V\subset W$ and consider any $U\subsetneqq V$. Let $v\in V$ and note that $\partial_{U-\{u\}}^V\capX$ can be written as the composition of cubes
\begin{align*}
\xymatrix{
  \partial_{U-\{u\}}^{V-\{u\}}\capX\ar[r] &
  \partial_U^V\capX
}
\end{align*}
We know by the induction assumption that the composition of cubes is $k_{(V-U)\cup\{u\}}$-cartesian and the right-hand cube is $k_{V-U}$-cartesian. Since $k_{V-U}\leq k_{(V-U)\cup\{u\}}$ by assumption, it follows from Proposition \ref{prop:map_of_cubical_diagrams}(c) that the left-hand cube is $k_{V-U}$-cartesian which finishes the proof; note that the sets $(V-\{u\})-(U-\{u\})$ and $V-U$ are the same.
\end{proof}

The following proposition will be needed in the proof of Proposition \ref{prop:properties_of_the_X_sup_A_construction} below.

\begin{prop}
\label{prop:lim_over_punctured_cube_of_X_sup_A_construction}
Let $\capO$ be an operad in $\capR$-modules and $W$ a finite set. Let $\capX$ be a $W$-cube of $\capO$-algebras (resp. left $\capO$-modules) and consider any concave subset $\capA\subset\powerset(W)$. Then for each nonempty subset $V\subset W$, there is a natural isomorphism
\begin{align*}
  \lim_{\powerset_0(W-V)}\partial_V^W\capX^\capA\Iso
  \lim_{T\in\capA,\,T\supsetneqq V}\capX_T
\end{align*}
\end{prop}

\begin{proof}
This is because the indexing sets $\{T\in\capA:T\supset U,\, U\supsetneqq V\}$ and $\{T\in\capA:T\supsetneqq V\}$ are the same.
\end{proof}

The following proposition explains the key properties of the $\capX^\capA$ construction and its relationship to $\capX$; it is through these properties that the $\capX^\capA$ construction is useful and meaningful.

\begin{prop}
\label{prop:properties_of_the_X_sup_A_construction}
Let $\capO$ be an operad in $\capR$-modules and $W$ a nonempty finite set. Let $\capX$ be a $W$-cube of $\capO$-algebras (resp. left $\capO$-modules) and consider any concave subset $\capA_W^\mathrm{min}\subset\capA\subset\capA_W^\mathrm{max}$.
\begin{itemize}
\item[(a)] There are natural isomorphisms $\partial_V^W\capX^\capA\Iso\partial_V^W\capX$ if $V\in\capA$,
\item[(b)] The $(W-V)$-cube $\partial_V^W\capX^\capA$ is $\infty$-cartesian if $V\notin\capA$ and $\capX$ is a fibration $W$-cube of $\capO$-algebras (resp. left $\capO$-modules).
\end{itemize}
\end{prop}

\begin{proof}
Consider part (a). Let $V\in\capA$. Then we know that the indexing sets $\{T\in\capA:T\supset V\}$ and $\{T:W\supset T\supset V\}$ are the same, by $\capA\subset\capA_W^\mathrm{max}$ concave. It follows that $\capX_V\Iso\lim_{W\supset T\supset V}\capX_T=\lim_{T\in\capA,\,T\supset V}=(\capX^\capA)_V$. The same argument shows that $\capX_{V'}\Iso(\capX^\capA)_{V'}$, for each $V'\supset V$. Consider part (b). Let $V\subset W$ and $V\notin\capA$. Then the indexing sets $\{T\in\capA:T\supset V\}$ and $\{T\in\capA:T\supsetneqq V\}$ are the same, and hence the composition
\begin{align*}
  (\partial_V^W\capX^\capA)_\emptyset=(\capX^\capA)_V=\lim_{T\in\capA,\,T\supset V}X_T=\lim_{T\in\capA,\,T\supsetneqq V}X_T\Iso\lim_{\powerset_0(W-V)}\partial_V^W\capX^\capA
\end{align*}
is an isomorphism which finishes the proof of part (b).
\end{proof}

The following proposition shows that $\capX^\capA$ inherits several of the homotopical properties of $\capX$.

\begin{prop}
\label{prop:cartesian_estimates_for_induction_argument}
Let $\capO$ be an operad in $\capR$-modules and $W$ a nonempty finite set. Let $\capX$ be a fibration $W$-cube of $\capO$-algebras (resp. left $\capO$-modules) and consider any concave subset $\capA_W^\mathrm{min}\subset\capA\subset\capA_W^\mathrm{max}$. Assume that
\begin{itemize}
\item[(i)] for each nonempty subset $V\subset W$, the $V$-cube $\partial_{W-V}^W\capX$ (formed by all maps in $\capX$ between $\capX_{W-V}$ and $\capX_W$) is $k_V$-cartesian,
\item[(ii)] $k_{U}\leq k_V$ for each $U\subset V$.
\end{itemize}
Then, for each $U\subsetneqq V\subset W$, the $(V-U)$-cube $\partial_U^V\capX^\capA$ is $k_{V-U}$-cartesian.
\end{prop}

\begin{proof}
By the downward induction argument in the proof of Proposition \ref{prop:warmup_cartesian_estimates_for_induction_argument}, it suffices to verify that the $V$-cube $\partial_{W-V}^W\capX^\capA$ is $k_V$-cartesian, for each nonempty subset $V\subset W$, and this follows immediately from Proposition \ref{prop:properties_of_the_X_sup_A_construction}.
\end{proof}

The following proposition appears in Goodwillie \cite[2.8]{Goodwillie_calc2} in the context of spaces, and exactly the same argument gives a proof in the context of structured ring spectra; this is an exercise left to the reader.

\begin{prop}
\label{prop:fibration_cube_properties}
Let $\capO$ be an operad in $\capR$-modules and $W$ a nonempty finite set. Let $\capX$ be a fibration $W$-cube of $\capO$-algebras (resp. left $\capO$-modules) and consider any concave subset $\capA\subset\powerset(W)$.
\begin{itemize}
\item[(a)] For each inclusion $\capA'\subset\capA$ of concave subsets of $\powerset(W)$, the induced map $\lim_{\capA}\capX\rarrow\lim_{\capA'}\capX$ is a fibration in $\AlgO$ (resp. $\LtO$),
\item[(b)] For any concave subsets $\capA,\capB$ of $\powerset(W)$, the diagram
\begin{align*}
\xymatrix{
  \lim_{\capA\cup\capB}\capX\ar[r]\ar[d] & \lim_{\capB}\capX\ar[d]\\
  \lim_{\capA}\capX\ar[r] & \lim_{\capA\cap\capB}\capX
}
\end{align*}
is a pullback diagram of fibrations in $\AlgO$ (resp. $\LtO$).
\item[(c)] If $A\in\capA$ is minimal and the fibration $$\capX_A\Iso\lim_{\powerset(W-A)}\partial_A^W\capX\rarrow\lim_{\powerset_0(W-A)}\partial_A^W\capX$$
is $k_{W-A}$-connected, then the fibration $$\lim_\capA\capX\rarrow\lim_{\capA-\{A\}}\capX$$ is $k_{W-A}$-connected.
\end{itemize}
\end{prop}

The purpose of the following induction argument is to leverage the higher dual homotopy excision result (Theorem \ref{thm:higher_dual_homotopy_excision}) for structured ring spectra into a conceptual proof of the second main theorem in this paper (Theorem \ref{thm:higher_dual_blakers_massey})---the higher dual Blakers-Massey theorem for structured ring spectra. Proposition \ref{prop:dual_cubical_induction_argument} is motivated by Goodwillie \cite[proof of (2.6)]{Goodwillie_calc2}; it is essentially Goodwillie's dual cubical induction argument, appropriately modified to our situation. The reader who is interested in an alternate proof of Theorem \ref{thm:higher_dual_blakers_massey}, which is more efficient, but requires a little extra calculation at the end, may skip directly to Remark \ref{rem:alt_proof_of_higher_dual_blakers_massey}.

\begin{prop}[Dual cubical induction argument]
\label{prop:dual_cubical_induction_argument}
Let $\capO$ be an operad in $\capR$-modules and $W$ a nonempty finite set. Suppose $\capA\subsetneqq\capA_W^\mathrm{max}$ is concave, $A\in\capA$ minimal, $|W-A|\geq 2$, and $\capA':=\capA-\{A\}$. Let $\capX$ be a fibration $W$-cube of $\capO$-algebras (resp. left $\capO$-modules). Assume that $\capR,\capO,\capX_\emptyset$ are $(-1)$-connected, and suppose that
\begin{itemize}
\item[(i)] for each nonempty subset $V\subset W$, the $V$-cube $\partial_{W-V}^W\capX$ (formed by all maps in $\capX$ between $\capX_{W-V}$ and $\capX_W$) is $k_V$-cartesian,
\item[(ii)] $-1\leq k_{U}\leq k_V$ for each $U\subset V$.
\end{itemize}
If $X^{\capA'}$ is $j$-cocartesian, then $\capX^\capA$ is $j$-cocartesian, where $j$ is the minimum of $|W|+\sum_{V\in\lambda}k_V$ over all partitions $\lambda$ of $W$ by nonempty sets not equal to $W$.
\end{prop}

\begin{rem}
In the case that $\capX$ is a $3$-cube, the following diagram illustrates one of the cubical decompositions covered by Proposition \ref{prop:dual_cubical_induction_argument}. It corresponds to the sequence of minimal elements: $\{3\}$, $\{2\}$, $\{1\}$, $\emptyset$.

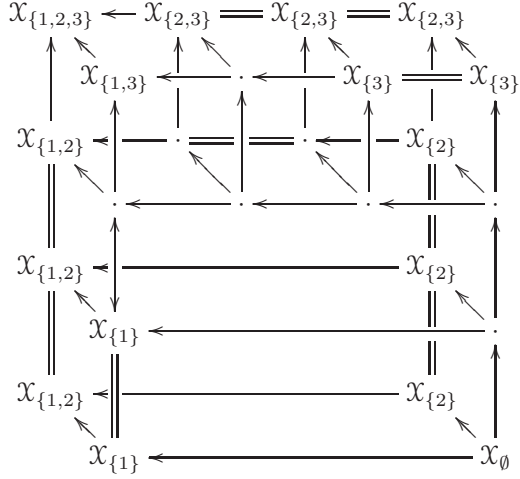
\begin{align*}
\xymatrix@!0{
\capX_{\{1,2,3\}} &&
\capX_{\{2,3\}}\ar[ll]\ar@{=}[rr] &&
\capX_{\{2,3\}}\ar@{=}[rr] &&
\capX_{\{2,3\}}\\
&\capX_{\{1,3\}}\ar[ul] &&
\cdot\ar[ll]\ar[ul] &&
\capX_{\{3\}}\ar@{=}[rr]\ar[ll]\ar[ul] &&
\capX_{\{3\}}\ar[ul]\\
\capX_{\{1,2\}}\ar@{=}[dd]\ar[uu] &&
\cdot\ar@{=}'[r][rr]\ar'[l][ll]\ar'[u][uu] &&
\cdot\ar'[u][uu] &&
\capX_{\{2\}}\ar@{=}'[d][dd]\ar'[l][ll]\ar[uu]|(0.47)\hole|(0.5)\hole\\
&\cdot\ar[dd]\ar[ul]\ar[uu] &&
\cdot\ar[ll]\ar[ul]\ar[uu] &&
\cdot\ar[ll]\ar[ul]\ar[uu] &&
\cdot\ar[ll]\ar[ul]\ar[uu]\\
\capX_{\{1,2\}}\ar@{=}[dd]&&&&&&
\capX_{\{2\}}\ar@{=}'[d][dd]\ar'[lllll][llllll]\\
&\capX_{\{1\}}\ar@{=}[dd]\ar[ul]\ar[uu]&&&&&&
\cdot\ar[llllll]\ar[ul]\ar[uu]\\
\capX_{\{1,2\}} &&&&&&
\capX_{\{2\}}\ar'[lllll][llllll]\\
&\capX_{\{1\}}\ar[ul]&&&&&&\capX_\emptyset\ar[llllll]\ar[ul]\ar[uu]
}
\end{align*}
\end{rem}

\begin{proof}[Proof of Proposition \ref{prop:dual_cubical_induction_argument}]
It suffices to consider the case of left $\capO$-modules. The argument is by induction on $|W|$. The case $|W|=2$ is verified by the proof of Theorem \ref{thm:dual_blakers_massey}. Let $n\geq 3$ and assume the proposition is true for each $2\leq|W|<n$. Let's verify it remains true for $|W|=n$.

We know by assumption that $\capX^{\capA'}$ is $j$-cocartesian in $\LtO$. We want to verify that $\capX^\capA$ is $j$-cocartesian in $\LtO$ (it might be helpful at this point to look ahead to \eqref{eq:dual_final_decomposition_for_induction_argument} for the decomposition of $\capX^\capA$ that we will use to finish the proof). Consider the induced map of $W$-cubes $\capX^{\capA}\rarrow\capX^{\capA'}$. Note that if $U\not\subset A$, then $\{T\in\capA:T\supset U\}=\{T\in\capA':T\supset U\}$; in particular, each of the maps
\begin{align}
\label{eq:dual_subdiagram_of_identity_maps}
\xymatrix{
  (\capX^{\capA})_U\ar[r]^-{\id} & (\capX^{\capA'})_U,
  & U\not\subset A
}
\end{align}
in $\capX^{\capA}\rarrow\capX^{\capA'}$ is the identity map. Note that if $U\subset V\subset A$, then the diagram
\begin{align}
\label{eq:pullback_diagram_in_induction_step}
\xymatrix{
  (\capX^{\capA})_U\ar[r]\ar[d] & (\capX^{\capA'})_U\ar[d]\\
  (\capX^{\capA})_V\ar[r] & (\capX^{\capA'})_V, &
  U\subset V\subset A
}
\end{align}
is a pullback diagram by Proposition \ref{prop:fibration_cube_properties}(b); in particular, focusing on this case is the same as focusing on the subdiagram $\partial_\emptyset^A\capX^\capA\rarrow\partial_\emptyset^A\capX^{\capA'}$ of $\capX^\capA\rarrow\capX^{\capA'}$.

Let's first verify that $\partial_\emptyset^A\capX^\capA\rarrow\partial_\emptyset^A\capX^{\capA'}$ is $(j+|A|+1-|W|)$-cocartesian in $\LtO$. Let $\capY$ denote $\partial_\emptyset^A\capX^\capA\rarrow\partial_\emptyset^A\capX^{\capA'}$ regarded as an $(A\cup\{*\})$-cube as follows:
\begin{align*}
  \capY_V&:=(\partial_\emptyset^A\capX^\capA)_V=(\capX^\capA)_V,\quad\ V\subset A,\\
  \capY_V&:=(\partial_\emptyset^A\capX^{\capA'})_U=(\capX^{\capA'})_U,\quad U\subset A,\quad V=U\cup\{*\}.
\end{align*}
Since $|A\cup\{*\}|<|W|$, our induction assumption can be applied to $\capY$, provided that the appropriate $k'_V$-cartesian estimates are satisfied. We claim that with the following definitions
\begin{align*}
  k'_V&:=k_V,\quad\quad \emptyset\neq V\subset A,\\
  k'_V&:=k_{W-A},\quad\quad V=\{*\}\\
  k'_V&:=\infty,\quad\quad \emptyset\neq U\subset A,\quad V=U\cup\{*\},
\end{align*}
the cube $\partial_{(A\cup\{*\})-V}^{A\cup\{*\}}\capY$ (formed by all maps in $\capY$ between $\capY_{(A\cup\{*\})-V}$ and $\capY_{A\cup\{*\}}$) is $k'_V$-cartesian, for each nonempty subset $V\subset A\cup\{*\}$. Let's verify this now: If $V\subset A$ is nonempty, then $\partial_{(A\cup\{*\})-V}^{A\cup\{*\}}\capY$ is the cube $\partial_{A-V}^A\capX^{\capA'}$ which is $k_V$-cartesian by Proposition \ref{prop:cartesian_estimates_for_induction_argument}. If $V=\{*\}$, then $\partial_{(A\cup\{*\})-V}^{A\cup\{*\}}\capY$ is the map $(\capX^\capA)_A\rarrow(\capX^{\capA'})_A$ which is $k_{W-A}$-connected by Proposition \ref{prop:fibration_cube_properties}(c). Finally, if $V=U\cup\{*\}$ for any nonempty $U\subset A$, then $\partial_{(A\cup\{*\})-V}^{A\cup\{*\}}\capY$ is the cube $\partial_{A-U}^A\capX^\capA\rarrow\partial_{A-U}^A\capX^{\capA'}$ which is $\infty$-cartesian by
\eqref{eq:pullback_diagram_in_induction_step} and Proposition \ref{prop:map_of_cubical_diagrams}(d). Hence by our induction hypothesis applied to $\capY$: since the sum $\sum_{V\in\lambda'}k'_V$ for a partition $\lambda'$ of $A\cup\{*\}$ by nonempty sets is always either $\infty$, or the sum $\sum_{U\in\lambda}k_U$ for a partition $\lambda$ of $W$ by nonempty sets in which some $U$ is $W-A$, we know that $\capY$ is $j'$-cocartesian in $\LtO$, where $j'$ is the minimum of $|A\cup\{*\}|+\sum_{U\in\lambda}k_U$ over all partitions $\lambda$ of $W$ by nonempty sets in which some $U$ is $W-A$. In particular, this implies that $\capY$, and hence $\partial_\emptyset^A\capX^{\capA}\rarrow\partial_\emptyset^A\capX^{\capA'}$, is $(j+|A|+1-|W|)$-cocartesian in $\LtO$.

We next want to verify that
\begin{align}
\label{eq:induction_argument_cocartesian_estimate_for_certain_subdiagrams}
  \partial_\emptyset^{A'}\capX^\capA\rarrow\partial_\emptyset^{A'}\capX^{\capA'}
  \quad \text{is $(j+|A'|+1-|W|)$-cocartesian in $\LtO$}
\end{align}
for each $A'\supset A$. We know that \eqref{eq:induction_argument_cocartesian_estimate_for_certain_subdiagrams} is true for $A'=A$ by above. We will argue by upward induction on $|A'|$. Suppose that \eqref{eq:induction_argument_cocartesian_estimate_for_certain_subdiagrams} is true for some $A'\supset A$. Let $a\in W-A'$ and note that the cube $\partial_\emptyset^{A'\cup\{a\}}\capX^{\capA}\rarrow\partial_\emptyset^{A'\cup\{a\}}\capX^{\capA'}$ can be written as the diagram of cubes
\begin{align}
\label{eq:upward_induction_cube_factorization}
\xymatrix{
  \partial_\emptyset^{A'}\capX^\capA\ar[r]\ar[d] &
  \partial_\emptyset^{A'}\capX^{\capA'}\ar[d]\\
  \partial_{\{a\}}^{A'\cup\{a\}}\capX^\capA\ar[r]^-{\id} &
  \partial_{\{a\}}^{A'\cup\{a\}}\capX^{\capA'}
}
\end{align}
We know the bottom arrow is the identity by \eqref{eq:dual_subdiagram_of_identity_maps}, and the top arrow is $(j+|A'|+1-|W|)$-cocartesian in $\LtO$ by assumption. It follows from  Proposition \ref{prop:map_of_cubical_diagrams}(b) that  \eqref{eq:upward_induction_cube_factorization} is $(j+|A'|+2-|W|)$-cocartesian in $\LtO$, which finishes the argument that \eqref{eq:induction_argument_cocartesian_estimate_for_certain_subdiagrams} is true for each $A'\supset A$.

To finish off the proof, let $a\in W-A$ and note by \eqref{eq:dual_subdiagram_of_identity_maps} that $\capX^\capA$ can be written as the composition of cubes
\begin{align}
\label{eq:dual_final_decomposition_for_induction_argument}
  \partial_\emptyset^{W-\{a\}}\capX^\capA\rarrow
  \partial_\emptyset^{W-\{a\}}\capX^{\capA'}\rarrow
  \partial_{\{a\}}^W\capX^{\capA'}.
\end{align}
The left-hand arrow is $j$-cocartesian in $\LtO$ by \eqref{eq:induction_argument_cocartesian_estimate_for_certain_subdiagrams} and the right-hand arrow is $\capX^{\capA'}$ which is $j$-cocartesian in $\LtO$ by assumption, hence by Proposition \ref{prop:composed_map_of_cubical_diagrams}(a) it follows that $\capX^\capA$ is $j$-cocartesian in $\LtO$ which finishes the proof.
\end{proof}

\begin{thm}[Theorem \ref{thm:higher_dual_blakers_massey} restated]
Let $\capO$ be an operad in $\capR$-modules and $W$ a nonempty finite set. Let $\capX$ be a $W$-cube of $\capO$-algebras (resp. left $\capO$-modules). Assume that $\capR,\capO,\capX_\emptyset$ are $(-1)$-connected, and suppose that
\begin{itemize}
\item[(i)] for each nonempty subset $V\subset W$, the $V$-cube $\partial_{W-V}^W\capX$ (formed by all maps in $\capX$ between $\capX_{W-V}$ and $\capX_W$) is $k_V$-cartesian,
\item[(ii)] $-1\leq k_{U}\leq k_V$ for each $U\subset V$.
\end{itemize}
Then $\capX$ is $k$-cocartesian, where $k$ is the minimum of $k_W+|W|-1$ and $|W|+\sum_{V\in\lambda}k_V$ over all partitions $\lambda$ of $W$ by nonempty sets not equal to $W$.
\end{thm}

\begin{proof}
It suffices to consider the case of left $\capO$-modules. It is enough to consider the special case where $\capX$ is a fibration $W$-cube in $\LtO$. Let $\capA:=\capA^{\mathrm{max}}_W$, $\capA':=\capA^{\mathrm{max}}_W-\emptyset$, and note by \eqref{eq:dual_subdiagram_of_identity_maps} that $\capX^\capA$, which is equal to $\capX$, can be written as the composition of cubes
\begin{align}
\label{eq:dual_final_decomposition_dual_higher_homotopy_excision}
  \partial_\emptyset^{W-\{a\}}\capX^\capA\rarrow
  \partial_\emptyset^{W-\{a\}}\capX^{\capA'}\rarrow
  \partial_{\{a\}}^W\capX^{\capA'}.
\end{align}
We know by Proposition \ref{prop:dual_cubical_induction_argument}, together with Theorem \ref{thm:higher_dual_homotopy_excision} (to start the induction using $\capA'=\capA_W^{\mathrm{min}}$) that the right-hand arrow, which is $\capX^{\capA'}$, is $j$-cocartesian in $\AlgO$, where $j$ is the minimum of $|W|+\sum_{V\in\lambda}k_V$ over all partitions $\lambda$ of $W$ by nonempty sets not equal to $W$. We claim that the left-hand arrow is $(k_W+|W|-1)$-cocartesian in $\LtO$; this follows from upward induction by arguing exactly as in \eqref{eq:upward_induction_cube_factorization}, but by starting with the observation (see the $k'_V$ estimates in the proof of Proposition \ref{prop:dual_cubical_induction_argument}) that the map $\partial_\emptyset^\emptyset\capX^\capA\rarrow\partial_\emptyset^\emptyset\capX^{\capA'}$, which is the map $(\capX^\capA)_\emptyset\rarrow(\capX^{\capA'})_\emptyset$, is $k_W$-connected. Hence it follows from Proposition \ref{prop:composed_map_of_cubical_diagrams} that the composition,
which is $\capX$, is $k$-cocartesian in $\LtO$ which finishes the proof.
\end{proof}

\begin{rem}
\label{rem:alt_proof_of_higher_dual_blakers_massey}
An alternate proof of Theorem \ref{thm:higher_dual_blakers_massey} can be obtained from the higher Blakers-Massey theorem (Theorem \ref{thm:higher_blakers_massey}) by arguing as in the proof of Theorem \ref{thm:higher_dual_homotopy_excision}. Since the dual cubical induction argument is conceptually very useful and will be needed elsewhere, we include both approaches for the interested reader, with only a few details of the alternate proof below left to the reader.
\end{rem}

\begin{proof}
It suffices to consider the case of left $\capO$-modules. The argument is by induction on $|W|$. The case $|W|$=1 is trivial and the case $|W|=2$ is verified by Theorem \ref{thm:dual_blakers_massey}. Let $n\geq 3$ and assume the theorem is true for each $2\leq |W|< n$. Let's verify it remains true for $|W|=n$.

It suffices to consider the special case where $\capX$ is a cofibration $W$-cube in $\LtO$. Consider the induced maps
\begin{align}
\label{eq:alt_composition_of_maps_for_determining_cocartesian_estimate}
\xymatrix{
  \colim\nolimits^{\SymSeq}_{\powerset_1(W)}\capX\ar[r]^-{(*)} &
  \colim\nolimits^{\LtO}_{\powerset_1(W)}\capX\ar[r]^-{(**)} &
  \colim\nolimits^{\LtO}_{\powerset(W)}\Iso\capX_W
}
\end{align}
We want to show that $(**)$ is $k$-connected. Consider the $W$-cube $\capX':=\capX_{\capA^W_\mathrm{max}-\{W\}}$ in $\LtO$ and note that $(*)$ is isomorphic to the map
\begin{align*}
\xymatrix{
  \colim\nolimits^{\SymSeq}_{\powerset_1(W)}\capX'\ar[r]^-{(*)'} &
  \colim\nolimits^{\SymSeq}_{\powerset(W)}\capX'\Iso\capX'_W
}
\end{align*}
We know that $\capX'$ is $\infty$-cocartesian in $\LtO$. By Proposition \ref{prop:warmup_cartesian_estimates_for_induction_argument} we know that the $V$-cube $\partial_U^{U\cup V}\capX$ is $k_V$-cartesian for each disjoint $U$ and $V$. Hence by the induction hypothesis, for each $V\subsetneqq W$ the cube $\partial_\emptyset^V\capX$ is $k'_V$-cocartesian, where $k'_V$ is the minimum of $k_V+|V|-1$ and $|V|+\sum_{U\in\lambda'}k_U$ over all partitions $\lambda'$ of $V$ by nonempty sets not equal to $V$. In particular, the $W$-cube $\capX'$ satisfies the conditions of the higher Blakers-Massey theorem (Theorem \ref{thm:higher_blakers_massey}) with $k'_W=\infty$ and the $k'_V$ above, and hence it follows that $\capX'$ is $k'$-cartesian, where $k'$ is the minimum of $-|W|+\sum_{V\in\lambda}(k'_V+1)$ over all partitions $\lambda$ of $W$ by nonempty sets not equal to $W$.

Hence by Proposition \ref{prop:comparing_cocartesian_and_cartesian_estimates_in_ModR} we know that $\capX'$ is $(k'+|W|-1)$-cocartesian in $\SymSeq$, and therefore $(*)$ is $(k'+|W|-1)$-connected. Since the composition in \eqref{eq:alt_composition_of_maps_for_determining_cocartesian_estimate} is $(k_W+|W|-1)$-connected, it follows from Proposition \ref{prop:connectivity_estimates_for_composition_of_maps}(b) that $(**)$ is $k$-connected, where $k$ is the minimum of $k_W+|W|-1$ and $\sum_{V\in\lambda}(k'_V+1)$ over all partitions $\lambda$ of $W$ by nonempty sets not equal to $W$; it is an exercise left to the reader to verify that this description of $k$ agrees with Theorem \ref{thm:higher_dual_blakers_massey}.
\end{proof}

\bibliographystyle{plain}
\bibliography{HigherHomotopyExcision}

\end{document}